\documentclass[12pt]{amsart}
\usepackage{textcomp}
\usepackage[latin9]{inputenc}
\usepackage{varwidth}
\usepackage{amsthm}
\usepackage{amssymb}
\usepackage{stmaryrd}
\usepackage{geometry}
\usepackage{xcolor}
\usepackage{comment}
\usepackage[bookmarks=false,
 breaklinks=false,pdfborder={0 0 1},colorlinks=false]
 {hyperref}

\makeatletter

\providecommand{\tabularnewline}{\\}
\newenvironment{cellvarwidth}[1][t]
    {\begin{varwidth}[#1]{\linewidth}}
    {\@finalstrut\@arstrutbox\end{varwidth}}

\numberwithin{equation}{section}
\numberwithin{figure}{section}
\theoremstyle{plain}
\newtheorem{thm}{Theorem}[section]
\newtheorem{prop}[thm]{Proposition}

\makeatother

\newtheorem{maintheorem}{Theorem}[]

\begin{document}
\title{Poincar\'{e} Polynomials and Curvature Operators of Symmetric Spaces}
\author{Peter Petersen}
\address{University of California, Los Angeles, 520 Portola Plaza, CA, 90095}
\email{petersen@math.ucla.edu}
\author{James Stanfield}
 \address{Cluster of Excellence Mathematics M\"{u}nster, Orl\'{e}ans-Ring 10, 48149 M\"{u}nster, Germany}
 \email{james.stanfield@uni-muenster.de} 
\begin{abstract}
We compute explicit formulas for the curvature operators and Poincar\'e polynomials of all compact irreducible symmetric spaces. We can easily derive the Poincar\'e polynomials using quantum numbers, giving a formula that mirrors the known formula for the Euler characteristic. For the curvature operators, we show that their maximum eigenvalue is always bounded above by the Einstein constant, with equality attained precisely by the Hermitian symmetric spaces.
\end{abstract}

\maketitle

\section{Introduction}

In this note we present an overview of symmetric spaces of compact type with an emphasis on showing how to calculate their Poincar\'e polynomials and
curvature operators. Although the Poincar\'e polynomials---in particular Betti numbers---can be found in the literature, it is perhaps not as straightforward as one might expect. Similarly, the curvature operators have been calculated in some cases, for instance on the Hermitian symmetric spaces \cite{Borel1960,CV}, but there remain gaps in the general setting. On the other hand, the eigenvalues for the curvature operator of the \emph{second kind} (that is, acting on symmetric tensors) were computed by Koiso on the irreducible symmetric spaces in \cite{Koiso1978}. To fill these gaps, we have compiled a complete list of Poincar\'e polynomials and curvature operators for each of the compact irreducible symmetric spaces. We have organized the spaces into five classes: Lie groups (see Subsection \ref{subsec:The-Groups}),
Grassmannians (see Subsection \ref{subsec:Grassmanian-Symmetric}),
spaces with simple isotropy group (see Subsection \ref{subsec:Simple-isotropy}),
Hermitian symmetric spaces (see Subsection \ref{subsec:Hermitian-Symmetric}),
and Wolf spaces (quaternionic-K\"ahler symmetric spaces) (see Subsection \ref{subsec:Wolf-spaces}). 

We consider compact simply-connected symmetric spaces: $M=\mathrm{G}/\mathrm{H}$, where
$\mathrm{G}$ is a compact semi-simple Lie group and $\mathrm{H}\subset \mathrm{G}$ is the identity component of the fixed point subgroup of an involutive automorphism of $\mathrm{G}$. On the Lie algebra $\mathfrak g$ of $\mathrm G$, the Killing form
$B_{\mathfrak{g}}$ is negative definite and we equip
$\mathrm{G}/\mathrm{H}$ with the $\mathrm{G}$-invariant submersion metric induced by $-B_{\mathfrak{g}}$. With this metric, $M$ is Einstein with Einstein constant $\frac12$. Moreover, the Lie algebra of Killing fields on $M$ is isomorphic to $\mathfrak{g}$. In particular, $\mathrm G$ acts almost-effectively on $M$.

Our first result addresses the nature of the eigenvalues of the curvature operator.

\begin{maintheorem}
\label{thm:CurvOp}
Any compact symmetric space has the property that the eigenvalues of the curvature operator lie in $[0,\frac{1}{2}]$. Moreover, if the space is irreducible there can be one, two, or three distinct nonzero eigenvalues that are all rational numbers. Finally, only the complex Grassmannians can have three such eigenvalues.

\end{maintheorem}
The first statement has a conceptual proof that also leads to a general principle for calculating the eigenvalues (see Subsection \ref{subsec:Methods-for-Calculating}).   The last observation follows from the tabulations of eigenvalues in section \ref{sec:The-Tables}. At the end of Subsection \ref{subsec:Methods-for-Calculating}, we compute the curvature operator for Jensen's second homogeneous Einstein metric on $\mathrm S^3 \times \mathrm S^3$, and in particular show that the maximum eigenvalue is bounded by the Einstein constant, hence this property can hold on non-symmetric spaces. We are not aware of any non-symmetric \emph{K\"ahler}--Einstein metric satisfying this bound.

When the maximum eigenvalue equals the Einstein constant, the space is forced to be Hermitian.
\begin{maintheorem}
\label{thm:Hermitian}The following are equivalent for a compact irreducible symmetric space $M=\mathrm{G}/\mathrm{H}$:
\begin{enumerate}
    \item The maximum eigenvalue for the curvature operator is $\frac12$,
    \item $\mathfrak{h}$ has a center of dimension $1$,
    \item $M$  is a Hermitian symmetric space,
    \item $b_2=1$.

\end{enumerate}
The eigenvector for $\frac12$  corresponds precisely to the K\"{a}hler form.
\end{maintheorem}
Here the equivalence of the last three statements is part of Cartan's classification of symmetric spaces. The proof uses a general formula for the curvature operator (see Proposition \ref{prop:curv_op}) to prove that (1) and (2) are equivalent. The fact that (1) implies (4) also follows from Theorem \ref{thm:Bochner}, which holds for Einstein manifolds without any symmetry assumptions.

By going through the table for compact irreducible symmetric spaces one can see that those with vanishing Euler characteristic all have simple isotropy group (see \ref{subsec:Simple-isotropy}). For spaces with simple isotropy we make the following observations.

\begin{maintheorem}
\label{thm:simpleisotropy}
If $M=\mathrm{G}/\mathrm{H}$ where $\mathrm{H}$  is simple, then there is only one nonzero eigenvalue 
\[\lambda =\frac{1}{4}\frac{\dim M}{\dim \mathrm{H} }.\]
Moreover, either
\begin{enumerate}
    \item $\dim M=\dim \mathrm{H}$ in which case $\mathrm{G}=\mathrm{H}\times \mathrm{H}$, i.e., $M$  is of type II and $b_3=1$, or
    \item  $\dim M \neq \dim \mathrm{H}$, in which case $\mathrm{G}$ is simple and $b_2 = b_3 =b_4=0$.
\end{enumerate}
\end{maintheorem}
The proof follows from the calculations in Subsections \ref{subsec:The-Groups} and \ref{subsec:Simple-isotropy}.

As mentioned, the Poincar\'{e} polynomials are well-known and were first
compiled by Takeuchi \cite{Takeuchi1962}, who discovered a general formula for symmetric
spaces (see also \cite{Iwamoto1949}). That said, the Poincar\'{e} polynomial for the real Grassmannian
$\mathrm{D}_{k+l}\mathrm{I}_{2k}$ (see Section \ref{sec:The-Tables}
for notation) seems to be missing a factor in our scanned copy of the paper.
All of the polynomials have been recalculated here using quantum numbers
as Takeuchi's formula becomes particularly elegant and easy to manipulate
with their use (see Subsections \ref{subsec:Quantum-Numbers} and \ref{subsec:Formulas-for-Poincare}). This also leads to a simple
formula for the Euler characteristics. There is an older formula due to Wang  (see \cite{Wang1949})  where the Euler characteristic is given as the index of the Weyl groups $W(\mathrm{H})\subset W(\mathrm{G})$. This is the formula used by Takeuchi. It
should be mentioned that there are unfortunately differences in both
Euler characteristics and Poincar\'{e} polynomials scattered throughout
the literature. The de Rham cohomologies with $\mathbb C$-coefficients for many symmetric spaces were computed by Leung in \cite{Leung2016}. For generalizations of Takeuchi's work to a larger class of homogeneous spaces see \cite{Kuin1965}, \cite{Terzic2001FPM}, and \cite{Terzic2001}. For a list of references as well as overview of the cohomology rings of exceptional symmetric spaces see \cite{Piccinni2017}. Lie groups and the classical symmetric spaces are covered in the comprehensive guide \cite{MinuraToda1991} (see also \cite{Ramanujam1969} for an elegant derivation using Morse theory). For a treatment of the cohomology of manifolds and homogeneous spaces see also \cite{GHV}.

The precise subgroups $\mathrm{H}\subset \mathrm{G}$ have also been selected from Takeuchi
with the exception of $\mathrm{E}_{8}\mathrm{VIII}$. The literature
has discrepancies in what these subgroups are. The best guide is \cite{Yokota2025}, where all the relevant subgroups of the exceptional groups are explained in detail. 

\subsection*{Acknowledgments}
The authors would like to thank Yuri Nikolaevski, Jan Nienhaus, Matthias Wink, Ramiro Lafuente, and Tom\'as Otero for helpful discussions as well as MATRIX for facilitating this work.

\section{Notation and Methods}\label{sec:Notation-and-Methods}

\subsection{Quantum Numbers}\label{subsec:Quantum-Numbers}

For positive integers $n$ the quantum numbers are the partial sums
of a geometric progression: 
\[
\left[n\right]_{q}=1+q+\cdots+q^{n-1}=\frac{1-q^{n}}{1-q}.
\]
If $m\mid n$, then 
\[
\frac{\left[n\right]_{q}}{\left[m\right]_{q}}=\left[\frac{n}{m}\right]_{q^{m}}.
\]
The quantum binomial also known as the Gaussian binomial is given
by 
\[
\binom{k+l}{k}_{q}=\frac{\left[k+l\right]_{q}\cdots\left[1\right]_{q}}{\left[k\right]_{q}\cdots\left[1\right]_{q}\left[l\right]_{q}\cdots\left[1\right]_{q}}=\frac{\left[k+l\right]_{q}\cdots\left[l+1\right]_{q}}{\left[k\right]_{q}\cdots\left[1\right]_{q}}.
\]
This is indeed a polynomial with nonnegative integer coefficients
as it satisfies the recurrence relations 
\[
\binom{k+l}{k}_{q}=q^{k}\binom{k+l-1}{k}_{q}+\binom{k+l-1}{k-1}_{q}
\]
and
\[
\binom{k+l}{k}_{q}=\binom{k+l-1}{k}_{q}+q^{l}\binom{k+l-1}{k-1}_{q}.
\]

This is not always completely obvious even in cases that are easy
to calculate. Consider, for example,
\[
\binom{n}{3}_{q}=\frac{\left[n\right]_{q}\left[n-1\right]_{q}\left[n-2\right]_{q}}{\left[3\right]_{q}\left[2\right]_{q}}
\]
where $n-2$ and $n$ are not divisible by either 2 or 3. For this
to be a polynomial, the expression 
\[
\frac{\left[n-1\right]_{q}}{\left[3\right]_{q}\left[2\right]_{q}}
\]
must also be a polynomial. This in turn requires that $n-1$ be divisible
by 6. This is clearly the case as it is even and one of the numbers
$n-2$, $n-1$, $n$ must be divisible by 3. However, some of the coefficients
will be negative as: 
\[
\frac{\left[6\right]_{q}}{\left[3\right]_{q}\left[2\right]_{q}}=\frac{\left[2\right]_{q^{3}}}{\left[2\right]_{q}}=\frac{1+q^{3}}{1+q}=\frac{1+q^{3}}{1+q}=1-q+q^{2}.
\]

Other quirky rules also come into play, e.g., when $m|n$, then 
\[
\frac{\left[n\right]_{t^{m}}}{\left[m\right]_{t^{n}}}=\frac{\frac{1-t^{mn}}{1-t^{m}}}{\frac{1-t^{nm}}{1-t^{n}}}=\frac{1-t^{n}}{1-t^{m}}=\left[\frac{n}{m}\right]_{t^{m}}.
\]

\subsection{Formulas for Poincar\'{e} Polynomials and Euler Characteristics}\label{subsec:Formulas-for-Poincare}

A Lie group $\mathrm{G}$ has Poincar\'{e} polynomial 
\[
\chi_{G}\left(t\right)=\underset{k\in D_{G}}{\prod}\left(1-t^{2k-1}\right)
\]
where $D_{G}$ consists of the degrees of the generators for the $\mathrm{G}$-invariant
polynomials (see \cite{BorelChevalley1955} and \cite{Coleman1958}). The invariant
polynomials for a group of rank $r$ can be represented as the symmetric
homogeneous polynomials on $\mathbb{R}^{r}$ that are invariant under
the Weyl group. $D$ has $r$ elements. For Lie groups these degrees
are given by:

\medskip{}

\noindent\begin{minipage}[t]{1\columnwidth}%
\centering

\begin{tabular}{|l|c|}
\hline 
Type/Space  & $D_{G}$\tabularnewline
\hline 
$\mathrm{A}_{n}$: $\mathrm{SU}\left(n+1\right)$  & $2,3,...,n+1$\tabularnewline
\hline 
$\mathrm{U}\left(n\right)$  & $1,2,...,n$\tabularnewline
\hline 
$\mathrm{B}_{n}$: $\mathrm{SO}\left(2n+1\right)$  & $2,4,...,2n$\tabularnewline
\hline 
$\mathrm{C}_{n}$: $\mathrm{Sp}\left(n\right)$  & $2,4,...,2n$\tabularnewline
\hline 
$\mathrm{D}_{n}$: $\mathrm{SO}\left(2n\right)$  & $n,2,4,...,2\left(n-1\right)$\tabularnewline
\hline 
$\mathrm{D}_{1}$: $\mathrm{SO}\left(2\right)$  & 1\tabularnewline
\hline 
$\mathrm{E}_{6}$  & $2,5,6,8,9,12$\tabularnewline
\hline 
$\mathrm{E}_{7}$  & $2,6,8,10,12,14,18$\tabularnewline
\hline 
$\mathrm{E}_{8}$  & $2,8,12,14,18,20,24,30$\tabularnewline
\hline 
$\mathrm{F}_{4}$  & $2,6,8,12$\tabularnewline
\hline 
$\mathrm{G}_{2}$  & $2,6$\tabularnewline
\hline 
\end{tabular}%
\end{minipage}

Borel's formula (\cite{Borel1953}) that uses Hirsch's method for calculating the cohomology
of a homogenous space $\mathrm{G}/\mathrm{H}$, where $\mathrm{G}$ and $\mathrm{H}$ have the same rank
is 
\[
\chi_{G/H}\left(t\right)=\frac{\underset{k\in D_{G}}{\prod}\left(1-t^{2k}\right)}{\underset{l\in D_{H}}{\prod}\left(1-t^{2l}\right)}=\frac{\underset{k\in D_{G}}{\prod}\left[k\right]_{t^{2}}}{\underset{l\in D_{H}}{\prod}\left[l\right]_{t^{2}}}
\]
and 
\[
\chi_{G/H}=\chi_{G/H}\left(-1\right)=\frac{\underset{k\in D_{G}}{\prod}k}{\underset{l\in D_{H}}{\prod}l}.
\]
This shows that such homogeneous spaces only have Betti numbers in
even degrees. The majority of symmetric spaces are of this type. All
other homogeneous spaces have zero Euler characteristic (see \cite{Wang1949}).
For the remaining spaces Takeuchi developed a more general formula \cite{Takeuchi1962}. Assume that the rank of $\mathrm{G}$ is $R$ and the rank of $\mathrm{H}$ is
$r<R$. $D_{G}$ can be divided into two sets $D_{G}^{1}$ and $D_{G}^{0}$,
where $D_{G}^{0}$ corresponds to the degrees of polynomials that
vanish when restricted to $\mathrm{H}$. For any homogeneous space where $D_{G}^{1}$
and $D_{H}$ have the same number of elements, $r$, the formula for
the Poincar\'{e} polynomial becomes: 
\[
\chi_{G/H}\left(t\right)=\frac{\underset{k\in D_{G}^{1}}{\prod}\left(1-t^{2k}\right)}{\underset{l\in D_{H}}{\prod}\left(1-t^{2l}\right)}\underset{k\in D_{G}^{0}}{\prod}\left(1+t^{2k-1}\right)=\frac{\underset{k\in D_{G}^{1}}{\prod}\left[k\right]_{t^{2}}}{\underset{l\in D_{H}}{\prod}\left[l\right]_{t^{2}}}\underset{k\in D_{G}^{0}}{\prod}\left(1+t^{2k-1}\right).
\]
Takeuchi shows that this formula is valid for all symmetric spaces
and he calculates in the process $D_{G}^{0}$. It is often not difficult
to guess what $D_{G}^{0}$ should be based on the knowledge of $D_{H}$.
Another useful constraint is that the rational expression:
\[
\frac{\underset{k\in D_{G}^{1}}{\prod}\left(1-t^{2k}\right)}{\underset{l\in D_{H}}{\prod}\left(1-t^{2l}\right)}
\]
should be a polynomial.

\subsection{The Curvature Operator}
On a Riemannian manifold $(M,g)$, we denote by $R$ the Riemann curvature tensor with the convention that $\operatorname{sec}(X,Y) = R(X,Y,X,Y)$ for any orthonormal pair $X,Y$. The Riemannian metric $g$ induces the usual tensor metric on all associated tensor bundles. In particular, for $X,Y$ orthonormal, we have $|X\wedge Y|^2 = 2$. We will also identify $\bigwedge^2 TM$ with $\mathfrak{so}(TM) = \mathfrak{so}(TM,g) \subset \operatorname{End}(M)$ via the map
\[
\sideset{}{^{2}}\bigwedge TM \to \mathfrak{so}(TM); \qquad X\wedge Y \mapsto (Z \mapsto g(X,Z)Y - g(Y,Z)X).
\]
Under this identification the metric on $\bigwedge^2TM$ is identified with the metric $-\operatorname{tr}(\cdot)(\cdot)$ on $\mathfrak{so}(TM,g)$. The curvature operator is the map $\mathfrak{R} \colon \bigwedge^2 TM \to \bigwedge^2 TM$ defined by
\[
g( \mathfrak{R}(X\wedge Y),Z\wedge W) :=2R(X,Y,Z,W); \qquad X,Y,Z,W \in TM.
\]

\subsection{Methods for Calculating Eigenvalues}\label{subsec:Methods-for-Calculating}

We will be considering a compact symmetric space of the
form $M = \mathrm{G}/\mathrm{H}$, where $\mathrm{G}$ is a compact semisimple Lie group and $\mathrm{H}$ is the identity component of the fixed point subgroup of an involutive automorphism of $\mathrm{G}$. We endow the compact Lie algebra $\mathfrak g$ with the inner product $-B_{\mathfrak g}$, where $B_{\mathfrak{g}}$ denotes \emph{Killing form} of $\mathfrak g$. Using this, we obtain an orthogonal decomposition $\mathfrak g = \mathfrak h \oplus \mathfrak p$, where $\mathfrak h$ is the Lie algebra of $\mathrm H$. It is well-known that $[\mathfrak h, \mathfrak h] \subset \mathfrak h$, $[\mathfrak h , \mathfrak p] \subset \mathfrak p$, and $[\mathfrak p, \mathfrak p] \subset \mathfrak h$. Via action fields, the tangent space $T_{\mathrm{H}}M$ is naturally identified with $\mathfrak p$. We endow the space $\mathrm{G}/\mathrm{H}$ with the symmetric Riemannian metric $g$ induced by the $\operatorname{Ad}(\mathrm{H})$-invariant restriction of $-B_{\mathfrak g}$ to $\mathfrak p \times \mathfrak p$.

The Lie algebra $\mathfrak h$ is identified with the \emph{holonomy algebra} in $\mathfrak{so}(\mathfrak p, -B_{\mathfrak{g}})$ via the map
\[
\mathfrak{h} \to \mathfrak{so}(TM) =\sideset{}{^{2}}\bigwedge TM; \qquad Z \mapsto \operatorname{ad}_Z|_{\mathfrak p},
\]
which is a well-defined Lie algebra morphism since the adjoint maps $\operatorname{ad}_Z$ preserve $\mathfrak p$ for all $Z \in \mathfrak h$.  It is not hard to check that this map is injective. Indeed, if $\operatorname{ad}_v|_{\mathfrak{p}} = 0$, then since $\operatorname{ad}_v = \operatorname{ad}_v|_{\mathfrak h}$ is a derivation, and $[\mathfrak p,\mathfrak p] = \mathfrak h$, we find that
\[
\operatorname{ad}_v(\mathfrak h) = \operatorname{ad}_v([\mathfrak p , \mathfrak p]) = 0,
\]
which contradicts semisimplicity of $\mathfrak g$.

Our choice of metric gives the following standard formulas for the
Ricci curvature and curvature tensor: 
\begin{eqnarray*}
\mathrm{Ric} & = & \frac{1}{2}g    =   -\frac{1}{2}B_{\mathfrak{g}}|_{\mathfrak{p}\times\mathfrak{p}},\\
R\left(X_1,Y_1,X_2,Y_2\right) & = & -B_{\mathfrak g}\left(\left[X_1,Y_1\right],\left[X_2,Y_2\right]\right).
\end{eqnarray*}
We define an auxiliary symmetric nonnegative operator $P:\mathfrak{h}\rightarrow\mathfrak{h}$ by 
\[
B_{\mathfrak{g}}\left(P\cdot,\cdot\right)|_{\mathfrak{h}\times\mathfrak{h}}=B_{\mathfrak{h}}\left(\cdot,\cdot\right).
\]

This can be used to calculate the curvature operator explicitly from the Killing forms.
\begin{prop}
\label{prop:curv_op}
Under the identification $\mathfrak h \simeq \operatorname{ad}_{\mathfrak h}|_{\mathfrak p} \subset \sideset{}{^{2}}\bigwedge TM$, the curvature operator of $g$ is given by
$\mathfrak{R}=\frac{1}{2}\left(I-P\right)$.

\end{prop}

\begin{proof}
    In our convention, we have $R(X,Y) = \operatorname{ad}_{[X,Y]}|_{\mathfrak p}$ for all $X,Y \in \mathfrak p$. It follows that
\[ \mathfrak R(X\wedge Y) = \operatorname{ad}_{[X,Y]}|_{\mathfrak{p}} \in \mathfrak{so}(\mathfrak p, -B_{\mathfrak g}) = \sideset{}{^{2}}\bigwedge \mathfrak p. \]
Now let $Z \colon \mathfrak p \times \mathfrak p \to \mathfrak h$ be the unique bilinear map defined by 
\[(X\wedge Y)_{\mathfrak h} = \operatorname{ad}_{Z(X,Y)}|_{\mathfrak p},\]
where $(\cdot)_{\mathfrak h}$ denotes orthogonal projection onto $\operatorname{ad}_{\mathfrak h}|_{\mathfrak p} \subset \bigwedge^2 \mathfrak p$. Since any element of the holonomy algebra can be written as a linear combination of these projected bivectors, the result will follow from the claim that
\[
        \frac 12 (I - P)Z(X,Y) = [X,Y],
 \]
for all $X,Y \in \mathfrak p$. Indeed, for any $W \in \mathfrak{h}$, we have by definition of the map $P$  that for $Z = Z(X,Y)$,
    \begin{align*}
        -B_{\mathfrak g}\left(\frac 12(I - P)Z,W\right) &
        = -\frac 12 B_{\mathfrak g}(Z,W) + \frac12 B_{\mathfrak h}(Z,W) \\&
        =  - \frac 12 \operatorname{tr}(\operatorname{ad}_{Z}|_{\mathfrak p}\operatorname{ad}_{W}|_{\mathfrak p})\\
        &
        = \frac 12 g( X\wedge Y , \operatorname{ad}_{W}|_{\mathfrak p}) \\&
        = g(\operatorname{ad}_W X,Y) \\&
        = -B_{\mathfrak g}(W,[X,Y]),
    \end{align*}
    which proves the claim.
\end{proof}

For this method to be useful one needs to know the Killing form explicitly.
This is best achieved by representing $\mathfrak{g}\subset\mathfrak{m}$,
where $\mathfrak{m}$ is a simple matrix algebra of compact type.
In this case $B_{\mathfrak{m}}\left(A,B\right)=c\mathrm{tr}\left(AB\right)$
for a constant $c$ that can be calculated by evaluating both sides
for a specific $H=A=B$. The same can then be repeated with any simple
subalgebra of $\mathfrak{m}$, and consequently, for any subalgebra. The Killing forms for exceptional Lie
groups with suitable representations in matrix algebras can be found
in \cite{Yokota2025}.

A simpler method can often be used to calculate the eigenvalues. Since $\mathfrak h$ is the Lie algebra of a compact Lie group, it decomposes as $\mathfrak{h}=\mathfrak{h}_0\oplus_{i}\mathfrak{h}_{i}$ into simple factors $\mathfrak{h}_i, i\geq1$ and the center $\mathfrak{h}_0$.

\begin{prop}
The decomposition $\mathfrak{h}=\mathfrak{h}_0\oplus_{i}\mathfrak{h}_{i}$ is precisely the decomposition into eigenspaces for $\mathfrak{R}$ and the eigenvalue for $\mathfrak{h}_0$
is the Einstein constant $\frac12$. Consequently, if $\lambda_i$ is the eigenvalue corresponding to $\mathfrak{h}_i$ we obtain:
\[
\frac{1}{2}\dim_{\mathbb{R}}M=\mathrm{scal}=2\mathrm{tr}\mathfrak{R}=\dim \mathfrak{h}_0 +2\sum_{i\geq 1}\lambda_{i}\dim_{\mathbb{R}}\mathfrak{h}_{i}.
\]
\end{prop}

\begin{proof}
The curvature operator of $g$ is $\mathrm{G}$-invariant with image $\mathfrak{h}\subset\bigwedge^{2}T_{p}M$.
In the decomposition $\mathfrak{h}=\mathfrak{h}_0\oplus_{i}\mathfrak{h}_{i}$, each of the modules $\mathfrak h_i, i\geq1$ are pairwise inequivalent as $\operatorname{Ad}(\mathrm H)$ representations. Thus by Schur's lemma, $\mathfrak R |_{\mathfrak{h}_i}= \lambda_i I_{\mathfrak h_i}, i\geq 1$. The Killing form $B_{\mathfrak{h}}$ vanishes on the center $\mathfrak{h}_0$ . Therefore,  also $P$  vanishes on $\mathfrak{h}_0$ .  It then follows from proposition \ref{prop:curv_op} that $\mathfrak R |_{\mathfrak{h}_0}= \frac12 I_{\mathfrak h_0}$. This completes the proof.

\end{proof}

It can happen that $\lambda_i = \lambda_j$ for some $i\neq j$. The Wolf space $\mathrm D_{4}\mathrm I_4$, given by $ \mathrm{SO}\left(8\right)/\left(\mathrm{SO}\left(4\right)\times\mathrm{SO}\left(4\right)\right)$, has only one nonzero eigenvalue despite the fact that the holonomy
algebra has four irreducible factors. In contrast, the Wolf space $\mathrm G_2 I$ given by
$\mathrm{G}_{2}/\mathrm{SO}\left(4\right)$ has two distinct eigenvalues
(see Subsection \ref{subsec:Wolf-spaces}).

This formula can be used in the majority of cases as the holonomy is either simple or we are in a situation where there are two factors and
either the Hermitian or quaternionic structure dictates the value
of one of them. The Grassmannians are the only exceptions that are
not covered by this short-cut.

We finish this section by proving  Theorems \ref{thm:CurvOp} and \ref{thm:Hermitian} aside from observational data that appears in the tables (see Section \ref{sec:The-Tables}), where all the Poincar\'{e} polynomials and eigenvalues of the curvature operators of the compact irreducible symmetric spaces are listed.

\begin{proof}[Proof of Theorems \ref{thm:CurvOp} and \ref{thm:Hermitian}] By definition, $P$ is nonnegative. Moreover, since $\operatorname{ad}_Z|_{\mathfrak p}$ is skew-symmetric for all $Z \in \mathfrak h$, $B_{\mathfrak g}|_{\mathfrak h \times \mathfrak h} \leq B_{\mathfrak h}$, hence $P \leq I$. Consequently,
the eigenvalues for the curvature operator are between $0$ and $\frac12$. This proves the first claim in Theorem \ref{thm:CurvOp}. The remaining claims are covered by the calculations and tables in section \ref{sec:The-Tables}.

Moving on to Theorem \ref{thm:Hermitian}, recall that the equivalence of (2) - (4) is already known. By Proposition \ref{prop:curv_op}, the $\frac 12$ eigenspace of $\mathfrak R$ is precisely the kernel of $P$, which equals $\mathfrak h_0$, so (1) is equivalent to (2).
\end{proof}

The property of $\lambda_{\max}$ being no greater than the Einstein constant does not hold exclusively on symmetric spaces. Indeed, writing $\mathrm S^3 \times \mathrm S^3 = \mathrm{G}/\mathrm{H} = \mathrm{SU}(2)^3 / \Delta \mathrm{SU}(2)$, consider the standard homogeneous metric induced by the Killing form of $\mathrm{SU}(2)^3$. This is a nearly K\"ahler, Einstein metric with Einstein constant $\frac 5 {12}$. A reductive decomposition of $\mathfrak g$ is given by $\mathfrak g = \mathfrak h \oplus \mathfrak p$, where
\[
\mathfrak p = \{(X,Y,-X-Y) : X,Y \in \mathfrak{su}(2)\},
\]
and this splits orthogonally into irreducible summands as $\mathfrak p = \mathfrak p_1 \oplus \mathfrak p_2$, where
\[
\mathfrak p_1 := \{(X,-X,0):X \in \mathfrak{su}(2)\},\qquad \mathfrak p_2 := \{(X,X,-2X): X \in \mathfrak{su}(2)\}.
\]
As $\operatorname{Ad}(\mathrm H)$ representations, $\mathfrak p_1 = \mathfrak p_2 = \mathfrak{su}(2)$. The second wedge power now splits as
\[
\sideset{}{^{2}}\bigwedge \mathfrak p = \sideset{}{^{2}}\bigwedge \mathfrak p_1 \oplus \sideset{}{^{2}}\bigwedge \mathfrak p_2 \oplus \mathfrak p_1 \wedge \mathfrak p_2 = \mathfrak{su}(2) \oplus \mathfrak{su}(2) \oplus \mathfrak{su}(2) \oplus \operatorname{Sym}^2_0(\mathfrak{su}(2)) \oplus \mathbb R
\]
where we have identified $\bigwedge^2 \mathfrak p_i = \bigwedge^2 \mathfrak{su}(2) = \mathfrak{su}(2)$, and $\mathfrak p_1 \wedge \mathfrak p_2 = \mathfrak {su}(2) \otimes \mathfrak {su}(2) = \mathfrak{su}(2) \oplus \operatorname{Sym}^2_0(\mathfrak{su}(2)) \oplus \mathbb{R}$ as $\mathrm{SU}(2)$ representations. In this splitting $\mathfrak{R}$ takes the form
\[
\mathfrak R = \frac 1 {48}\begin{pmatrix}
    9&5&0&0&0\\ 5&9&0&0&0\\0&0&4&0&0\\0&0&0&-2&0\\0&0&0&0&4
\end{pmatrix},
\]
which has eigenvalues $\frac 7{24}$ with multiplicity $3$, $\frac 1{12}$ with multiplicity $7$, and $-\frac 1{24}$ with multiplicity $5$. Hence $\lambda_{\max}$ is bounded above by the Einstein constant.

It would be interesting to find similar examples of non-symmetric \emph{K\"ahler}--Einstein metrics where $\lambda_{\max}$ is the Einstein constant. By the computations in \cite{Yang94}, the K\"ahler--Einstein metric on the flag manifold $\mathrm{SU}(3)/\mathrm{T}^2$ has $\lambda_{\max}$ strictly greater than the Einstein constant.

Using a Bochner argument, we are able to determine the second Betti number under this curvature assumption.

\begin{thm} \label{thm:Bochner}
    Let $(M^n,g)$ be an Einstein manifold with $\operatorname{Ric} = \rho g$ and denote by $\lambda_1 \leq  \dots \leq \lambda_{\binom{n}{2}}$ the eigenvalues of the curvature operator.
    \begin{enumerate}
        \item If $\lambda_{\binom{n}{2}} < \rho$, then $b_2 = 0$.
        \item If $(M,g)$ is K\"{a}hler--Einstein and $\lambda_{\binom{n}{2} - 1} < \rho$, then $b_2 = 1$.
    \end{enumerate}
\end{thm}
\begin{proof}
    For both items, we use a Bochner argument. To that end, recall that for any harmonic $p$-form, it holds that
    \[ 0=g\left(\nabla^{*}\nabla\phi,\phi\right)+p\sum\mathrm{Ric}_{ij}\phi_{ii_{2}\cdots i_{p}}\phi_{ji_{2}\cdots i_{p}}-4\binom{p}{2}\sum_{i < j}\sum_{k < l} R_{ijkl}\phi_{iji_{3}\cdots i_{p}}\phi_{kli_{3}\cdots i_{p}}.
    \]
    We can think of $\phi \colon \bigwedge^2 TM \to \bigotimes^{p-2}TM$ as a map by setting $$\phi(e_i\wedge e_j) := 2\phi_{iji_3\dots i_p}e_{i_3} \otimes \dots \otimes e_{i_p}.$$
    The factor of $2$ ensures that $\operatorname{tr}(\phi ^t\phi) = |\phi|^2$. Then, using the definition of the curvature operator $g(\mathfrak R (e_i\wedge e_j),e_k \wedge e_l) = 2R_{ijkl}$ and the fact that $\{\frac{1}{\sqrt 2} e_i \wedge e_j\}_{i < j}$ is an orthonormal basis of $\bigwedge^2 TM$, we have
    \[
        4\binom{p}{2}\sum_{i < j}\sum_{k < l} R_{ijkl}\phi_{iji_{3}\cdots i_{p}}\phi_{kli_{3}\cdots i_{p}} = 2 \operatorname{tr}(\mathfrak R \circ \phi^t \circ \phi).
    \]
    Together with the Einstein condition this gives
    \[
        0 = \frac 1p g\left(\nabla^{*}\nabla\phi,\phi\right) + \rho |\phi|^2 - (p-1) \operatorname{tr}(\phi^t \phi \ \mathfrak R).
    \]
    Now beginning with item (1), if $p = 2$, and $\lambda_{\max} < \rho$, then since $\phi^t \circ \phi \geq 0$, we find that $g(\nabla^*\nabla \phi,\phi) \leq (\lambda_{\max} - \rho)|\phi|^2$, hence $\phi$ must vanish.

    For item (2), we recall that the K\"ahler form $\omega \in \bigwedge^2 TM$ is a harmonic $2$-form. It suffices to show that any harmonic $2$-form $\phi$ which is \emph{primitive} (i.e. $\phi(\omega) = 0)$ vanishes. Note that $\omega$ is itself an eigenvector of $\mathfrak R$, with eigenvalue $\rho$. Writing $\bigwedge^2 TM = \mathbb R \omega \oplus \mathfrak{su}(\frac n 2)$, we have for any primitive harmonic $2$-form $\phi$ that
    \[
    \operatorname{tr}(\phi^t \phi \ \mathfrak R) = \operatorname{tr}(\mathfrak{R}|_{\mathfrak{su}(\frac n2)} \circ \phi^t \phi) \leq \lambda_{{\binom{n}{2}} - 1}|\phi|^2,
    \]
    and the claim follows exactly as in item (1).
\end{proof}

\section{The Tables}\label{sec:The-Tables}

We now present the tables for the five classes of compact irreducible
symmetric spaces. Except for the Grassmannians each class has a special
formula that calculates the eigenvalues.

The notation for the symmetric spaces comes from Cartan's classification
of simple Lie groups and irreducible symmetric spaces. For the Lie
groups the notation is a Roman letter followed by a subscript as described
in the table above. For the symmetric spaces of type I there is in
addition a Roman numeral, e.g., $\mathrm{D}_{k}\mathrm{I}_{l}$, where
$\mathrm{D}_{k}$ refers to the transitive group acting on the space,
the Roman numeral is the subtype of the symmetric spaces with that
transitive action, and the final subscript is used to describe the
first factor of the isotropy for Grassmannians.

\subsection{The Groups or Type II Symmetric Spaces}\label{subsec:The-Groups}

These spaces consist of the compact simple Lie groups which as symmetric
spaces are written in the form $(\mathrm{G}\times\mathrm{G})/\Delta\mathrm{G}$,
where the isotropy is the diagonal subgroup. In particular, holonomy algebra is simple and there is only one nonzero eigenvalue
for the curvature operator. We additionally include the unitary groups
as they are more convenient to work with than the special unitary groups
when calculating the cohomology of the complex Grassmannians.

Since the isotropy is simple we have that 
\[
\frac{1}{2}\dim\mathrm{G}=2\lambda_{\mathfrak{g}}\dim\mathfrak{g}.
\]
Consequently 
\[
\lambda_{\mathfrak{g}}=\frac{1}{4}
\]
in all cases.

The Poincar\'{e} polynomials have been known for a long time (see \cite{Coleman1958}
for a history and elegant calculation of the relevant degrees). The
precise form used here was first used by Chevalley. These type II
examples are topologically characterized by the fact that $b_{3}\neq0$.
All other compact irreducible symmetric spaces with vanishing Euler
characteristic have $b_{3}=0$.

\newpage{}

\medskip{}

\noindent{}%
\noindent\begin{minipage}[t]{1\columnwidth}%
\begin{tabular}{|l|c|c|c|l|}
\hline 
Type/Space  & dim  & Eigenvalues/spaces  & $\chi_{G/H}\left(t\right)$  & $\chi$\tabularnewline
\hline 
\begin{cellvarwidth}[t]
\medskip{}

$\mathrm{A}_{n}$ : $\mathrm{SU}\left(n+1\right)$ 
\end{cellvarwidth} & \begin{cellvarwidth}[t]
\centering
 \medskip{}

$n\left(n+2\right)$ 
\end{cellvarwidth} & \begin{cellvarwidth}[t]
\medskip{}

$\lambda_{\mathfrak{g}}=\frac{1}{4}$ 
\end{cellvarwidth} & \begin{cellvarwidth}[t]
\centering
 \medskip{}

$\underset{k\in\left\{ 2,3,...,n+1\right\} }{\prod}\left(1+t^{2k-1}\right)$

\medskip{}
 
\end{cellvarwidth} & \begin{cellvarwidth}[t]
\medskip{}

0
\end{cellvarwidth}\tabularnewline
\hline 
\begin{cellvarwidth}[t]
\medskip{}

$\mathrm{U}\left(n\right)$ 
\end{cellvarwidth} & \begin{cellvarwidth}[t]
\centering
 \medskip{}

$n^{2}$ 
\end{cellvarwidth} & \begin{cellvarwidth}[t]
\medskip{}

$0_{\mathfrak{u}\left(1\right)}$,

$\lambda_{\mathfrak{su}\left(n\right)}=\frac{1}{4}$ 
\end{cellvarwidth} & \begin{cellvarwidth}[t]
\centering
 \medskip{}

$\underset{k\in\left\{ 1,2,...,n\right\} }{\prod}\left(1+t^{2k-1}\right)$

\medskip{}
 
\end{cellvarwidth} & \begin{cellvarwidth}[t]
\medskip{}

0
\end{cellvarwidth}\tabularnewline
\hline 
\begin{cellvarwidth}[t]
\medskip{}

$\mathrm{B}_{n}$ : $\mathrm{SO}\left(2n+1\right)$ 
\end{cellvarwidth} & \begin{cellvarwidth}[t]
\centering
 \medskip{}

$n\left(2n+1\right)$ 
\end{cellvarwidth} & \begin{cellvarwidth}[t]
\medskip{}

$\lambda_{\mathfrak{g}}=\frac{1}{4}$ 
\end{cellvarwidth} & \begin{cellvarwidth}[t]
\centering
 \medskip{}

$\underset{k\in\left\{ 2,4,...,2n\right\} }{\prod}\left(1+t^{2k-1}\right)$

\medskip{}
 
\end{cellvarwidth} & \begin{cellvarwidth}[t]
\medskip{}

0
\end{cellvarwidth}\tabularnewline
\hline  
\begin{cellvarwidth}[t]
\medskip{}

$\mathrm{C}_{n}$ : $\mathrm{Sp}\left(n\right)$ 
\end{cellvarwidth} & \begin{cellvarwidth}[t]
\centering
 \medskip{}

$n\left(2n+1\right)$ 
\end{cellvarwidth} & \begin{cellvarwidth}[t]
\medskip{}

$\lambda_{\mathfrak{g}}=\frac{1}{4}$ 
\end{cellvarwidth} & \begin{cellvarwidth}[t]
\centering
 \medskip{}

$\underset{k\in\left\{ 2,4,...,2n\right\} }{\prod}\left(1+t^{2k-1}\right)$

\medskip{}
 
\end{cellvarwidth} & \begin{cellvarwidth}[t]
\medskip{}

0
\end{cellvarwidth}\tabularnewline
\hline 
\begin{cellvarwidth}[t]
\medskip{}

$\mathrm{D}_{n}$ : $\mathrm{SO}\left(2n\right)$ 
\end{cellvarwidth} & \begin{cellvarwidth}[t]
\centering
 \medskip{}

$n\left(2n-1\right)$ 
\end{cellvarwidth} & \begin{cellvarwidth}[t]
\medskip{}

$\lambda_{\mathfrak{g}}=\frac{1}{4}$ 
\end{cellvarwidth} & \begin{cellvarwidth}[t]
\centering
 \medskip{}

$\underset{k\in\left\{ n,2,4,...,2\left(n-1\right)\right\} }{\prod}\left(1+t^{2k-1}\right)$

\medskip{}
 
\end{cellvarwidth} & \begin{cellvarwidth}[t]
\medskip{}

0
\end{cellvarwidth}\tabularnewline
\hline 
\begin{cellvarwidth}[t]
\medskip{}

$\mathrm{D}_{1}$ : $\mathrm{SO}\left(2\right)$ 

\medskip{}

\end{cellvarwidth} & \begin{cellvarwidth}[t]
\centering

 \medskip{}

$1$ 

\medskip{}

\end{cellvarwidth} & \begin{cellvarwidth}[t]
\medskip{}

$0_{\mathfrak{so}\left(2\right)}$ 

\medskip{}

\end{cellvarwidth} & \begin{cellvarwidth}[t]
\centering
 \medskip{}

$1+t$ 
\medskip{}
\end{cellvarwidth} & \begin{cellvarwidth}[t]
\medskip{}

0
\end{cellvarwidth}\tabularnewline
\hline 
\begin{cellvarwidth}[t]
\medskip{}

$\mathrm{E}_{6}$ 
\end{cellvarwidth} & \begin{cellvarwidth}[t]
\centering
 \medskip{}

78 
\end{cellvarwidth} & \begin{cellvarwidth}[t]
\medskip{}

$\lambda_{\mathfrak{g}}=\frac{1}{4}$ 
\end{cellvarwidth} & \begin{cellvarwidth}[t]
\centering
 \medskip{}

$\underset{k\in\left\{ 2,5,6,8,9,12\right\} }{\prod}\left(1+t^{2k-1}\right)$

\medskip{}
 
\end{cellvarwidth} & \begin{cellvarwidth}[t]
\medskip{}

0
\end{cellvarwidth}\tabularnewline
\hline 
\begin{cellvarwidth}[t]
\medskip{}

$\mathrm{E}_{7}$ 
\end{cellvarwidth} & \begin{cellvarwidth}[t]
\centering
 \medskip{}

133 
\end{cellvarwidth} & \begin{cellvarwidth}[t]
\medskip{}

$\lambda_{\mathfrak{g}}=\frac{1}{4}$ 
\end{cellvarwidth} & \begin{cellvarwidth}[t]
\centering
 \medskip{}

$\underset{k\in\left\{ 2,6,8,10,12,14,18\right\} }{\prod}\left(1+t^{2k-1}\right)$

\medskip{}
 
\end{cellvarwidth} & \begin{cellvarwidth}[t]
\medskip{}

0
\end{cellvarwidth}\tabularnewline
\hline 
\begin{cellvarwidth}[t]
\medskip{}

$\mathrm{E}_{8}$ 
\end{cellvarwidth} & \begin{cellvarwidth}[t]
\centering
 \medskip{}

248 
\end{cellvarwidth} & \begin{cellvarwidth}[t]
\medskip{}

$\lambda_{\mathfrak{g}}=\frac{1}{4}$ 
\end{cellvarwidth} & \begin{cellvarwidth}[t]
\centering
 \medskip{}

$\underset{k\in\left\{ 2,8,12,14,18,20,24,30\right\} }{\prod}\left(1+t^{2k-1}\right)$

\medskip{}
 
\end{cellvarwidth} & \begin{cellvarwidth}[t]
\medskip{}

0
\end{cellvarwidth}\tabularnewline
\hline 
\begin{cellvarwidth}[t]
\medskip{}

$\mathrm{F}_{4}$ 
\end{cellvarwidth} & \begin{cellvarwidth}[t]
\centering
 \medskip{}

52 
\end{cellvarwidth} & \begin{cellvarwidth}[t]
\medskip{}

$\lambda_{\mathfrak{g}}=\frac{1}{4}$ 
\end{cellvarwidth} & \begin{cellvarwidth}[t]
\centering
 \medskip{}

$\underset{k\in\left\{ 2,6,8,12\right\} }{\prod}\left(1+t^{2k-1}\right)$

\medskip{}
 
\end{cellvarwidth} & \begin{cellvarwidth}[t]
\medskip{}

0
\end{cellvarwidth}\tabularnewline
\hline 
\begin{cellvarwidth}[t]
\medskip{}

$\mathrm{G}_{2}$ 
\end{cellvarwidth} & \begin{cellvarwidth}[t]
\centering
 \medskip{}

14 
\end{cellvarwidth} & \begin{cellvarwidth}[t]
\medskip{}

$\lambda_{\mathfrak{g}}=\frac{1}{4}$ 
\end{cellvarwidth} & \begin{cellvarwidth}[t]
\centering
 \medskip{}

$\underset{k\in\left\{ 2,6\right\} }{\prod}\left(1+t^{2k-1}\right)$

\medskip{}
 
\end{cellvarwidth} &\begin{cellvarwidth}[t]
\centering
 \medskip{}
0

\end{cellvarwidth}\tabularnewline
\hline 
\end{tabular}%
\end{minipage}\newpage{}

\subsection{The Grassmannian Symmetric Spaces}\label{subsec:Grassmanian-Symmetric}

The spaces in this class consist of $p$ dimensional subspaces in
a $p+q$ dimensional vector module over $\mathbb{R}$, $\mathbb{C}$,
or $\mathbb{H}$. In the real case they are further divided into two
branches depending on the type of the isometry group, and when the
isometry group is $\mathrm{D}$ there is a further subdivision depending
on the type of the isotropy. The case where $p=1$ is of special interest
as they are the compact rank one symmetric spaces.

The curvature operator is calculated using the Killing forms which
are explicit and respect the inclusion $\mathfrak{h}\subset\mathfrak{g}$.

For $\mathfrak{g}$ we have that: 
\begin{eqnarray*}
B_{\mathfrak{su}\left(n\right)}\left(X,Y\right) & = & 2n\mathrm{tr}\left(XY\right),\\
B_{\mathfrak{sp}\left(n\right)}\left(X,Y\right) & = & 2\left(n+1\right)\mathrm{tr}\left(XY\right),\\
B_{\mathfrak{so}\left(n\right)}\left(X,Y\right) & = & \left(n-2\right)\mathrm{tr}\left(XY\right),
\end{eqnarray*}
and for $\mathfrak{h}$ 
\begin{eqnarray*}
B_{\mathfrak{u}\left(1\right)\oplus\mathfrak{su}\left(p\right)\oplus\mathfrak{su}\left(q\right)} & = & 0_{\mathfrak{u}\left(1\right)}\oplus B_{\mathfrak{su}\left(p\right)}\oplus B_{\mathfrak{su}\left(q\right)},\\
B_{\mathfrak{sp}\left(p\right)\oplus\mathfrak{sp}\left(q\right)}\left(X,Y\right) & = & B_{\mathfrak{sp}\left(p\right)}\oplus B_{\mathfrak{sp}\left(q\right)},\\
B_{\mathfrak{so}\left(p\right)\oplus\mathfrak{so}\left(q\right)}\left(X,Y\right) & = & B_{\mathfrak{so}\left(p\right)}\oplus B_{\mathfrak{so}\left(q\right)}.
\end{eqnarray*}
Thus $P$ in each of the three cases becomes: 
\begin{eqnarray*}
P & = & 0_{\mathfrak{u}\left(1\right)}\oplus\frac{p}{p+q}I_{\mathfrak{su}\left(p\right)}\oplus\frac{q}{p+q}I_{\mathfrak{su}\left(q\right)},\\
P & = & \frac{p+1}{p+q+1}I_{\mathfrak{sp}\left(p\right)}\oplus\frac{q+1}{p+q+1}I_{\mathfrak{sp}\left(q\right)},\\
P & = & \frac{p-2}{p+q-2}I_{\mathfrak{sp}\left(p\right)}\oplus\frac{q-2}{p+q-2}I_{\mathfrak{sp}\left(q\right).}
\end{eqnarray*}
This calculates the curvature operator 
\begin{eqnarray*}
\mathfrak{R} & = & \frac{1}{2}I_{\mathfrak{u}\left(1\right)}\oplus\frac{1}{2}\frac{q}{p+q}I_{\mathfrak{su}\left(p\right)}\oplus\frac{1}{2}\frac{p}{p+q}I_{\mathfrak{su}\left(q\right)},\\
\mathfrak{R} & = & \frac{1}{2}\frac{q}{p+q+1}I_{\mathfrak{sp}\left(p\right)}\oplus\frac{1}{2}\frac{p}{p+q+1}I_{\mathfrak{sp}\left(q\right)},\\
\mathfrak{R} & = & \frac{1}{2}\frac{q}{p+q-2}I_{\mathfrak{sp}\left(p\right)}\oplus\frac{1}{2}\frac{p}{p+q-2}I_{\mathfrak{sp}\left(q\right).}
\end{eqnarray*}

Alternately, one might argue, assuming $\mathfrak{h}_{0}=\mathfrak{u}\left(1\right)$,
that for symmetry reasons the ratio should be 
\[
\frac{\lambda_{\mathfrak{h}_{1}}}{\lambda_{\mathfrak{h}_{2}}}=\frac{q}{p}.
\]
This leads to the formulas 
\[
\frac{1}{2}\dim M=2\lambda_{\mathfrak{h}_{2}}\left(\frac{q}{p}\dim\mathfrak{h}_{1}+\dim\mathfrak{h}_{2}\right)
\]
and 
\[
\lambda_{\mathfrak{h}_{1}}=\frac{q}{p}\lambda_{\mathfrak{h}_{2}}.
\]

Borel's formula works in almost all cases without effort. We give
more details for the real Grassmannians as they have the most complex
behavior. When using $\mathbb{Z}_2$ however the real Grassmannians have formulas that are similar to those of complex and quaternionic Grassmannians (see \cite{Ozawa2022} for formulas that also use quantum binomials).
We start with the simplest case, including even dimensional spheres
($l=0$), where the ranks are equal: 
\[
\frac{\mathrm{SO}\left(2k+2l+1\right)}{\mathrm{SO}\left(2k\right)\times\mathrm{SO}\left(2l+1\right)}.
\]
The relevant degrees are: 
\begin{eqnarray*}
D_{\mathrm{SO}\left(2k+2l+1\right)} & = & \left\{ 2,4,...,2k+2l\right\} ,\\
D_{\mathrm{SO}\left(2k\right)} & = & \left\{ k,2,4,...,2\left(k-1\right)\right\} ,\\
D_{\mathrm{SO}\left(2l+1\right)} & = & \left\{ 2,4,...,2l\right\} 
\end{eqnarray*}
leading to 
\begin{eqnarray*}
\chi\left(t\right) & = & \frac{\left[2\right]_{t^{2}}\left[4\right]_{t^{2}}\cdots\left[2k+2l\right]_{t^{2}}}{\left[k\right]_{t^{2}}\left[2\right]_{t^{2}}\left[4\right]_{t^{2}}\cdots\left[2\left(k-1\right)\right]_{t^{2}}\left[2\right]_{t^{2}}\left[4\right]_{t^{2}}\cdots\left[2l\right]_{t^{2}}}\\
 & = & \frac{\left[2k\right]_{t^{2}}\left[1\right]_{t^{4}}\left[2\right]_{t^{4}}\cdots\left[k+l\right]_{t^{4}}}{\left[k\right]_{t^{2}}\left[1\right]_{t^{4}}\left[2\right]_{t^{4}}\cdots\left[k\right]_{t^{4}}\left[1\right]_{t^{4}}\left[2\right]_{t^{4}}\cdots\left[l\right]_{t^{4}}}\\
 & = & \binom{k+l}{k}_{t^{4}}\left[2\right]_{t^{2k}}.
\end{eqnarray*}

The other family with equal ranks is 
\[
\frac{\mathrm{SO}\left(2k+2l\right)}{\mathrm{SO}\left(2k\right)\times\mathrm{SO}\left(2l\right)}
\]
and can be calculated in a similar fashion. One annoying feature
is that the rational expression is not in a form that is obviously a polynomial.

In the final family that includes the odd dimensional spheres: 
\[
\frac{\mathrm{SO}\left(2k+2l+2\right)}{\mathrm{SO}\left(2k+1\right)\times\mathrm{SO}\left(2l+1\right)}
\]
the ranks are not equal. The degrees are given by 
\begin{eqnarray*}
D_{\mathrm{SO}\left(2k+2l+2\right)} & = & \left\{ k+l+1,2,4,...,2k+2l\right\} ,\\
D_{\mathrm{SO}\left(2k+1\right)} & = & \left\{ 2,4,...,2k\right\} ,\\
D_{\mathrm{SO}\left(2l+1\right)} & = & \left\{ 2,4,...,2l\right\} .
\end{eqnarray*}
The degree $k+l+1$ corresponds to the invariant polynomial that has
to vanish under restriction to $\mathrm{SO}\left(2k+1\right)\times\mathrm{SO}\left(2l+1\right)$.
The remaining terms form a Gaussian binomial as in the first case.

There are several special cases that are of interest. When $p=2$
we obtain Hermitian symmetric spaces and their Poincar\'{e} polynomials
are calculated in Subsection \ref{subsec:Hermitian-Symmetric}. When
$p=4$ they become Wolf or quaternion symmetric spaces and are covered
in Subsection \ref{subsec:Wolf-spaces}. Here we consider $p=3$ which
is a case that might be of independent interest.

First, the case where Borel's formula works: 
\[
\frac{\mathrm{SO}\left(3+2l\right)}{\mathrm{SO}\left(3\right)\times\mathrm{SO}\left(2l\right)}.
\]
Here 
\begin{eqnarray*}
\chi\left(t\right) & = & \binom{1+l}{1}_{t^{4}}\left[2\right]_{t^{2l}}\\
 & = & \left[2\right]_{t^{2l}}\left[l+1\right]_{t^{4}}.
\end{eqnarray*}

The other case is 
\[
\frac{\mathrm{SO}\left(3+2l+1\right)}{\mathrm{SO}\left(3\right)\times\mathrm{SO}\left(2l+1\right)}
\]
where 
\begin{eqnarray*}
\chi\left(t\right) & = & \binom{1+l}{1}_{t^{4}}\left(1+t^{2+2l+1}\right)\\
 & = & \left[1+l\right]_{t^{4}}\left(1+t^{2l+3}\right).
\end{eqnarray*}

\newpage{}

\medskip{}

\noindent{}%
\noindent\begin{minipage}[t]{1\columnwidth}%
\begin{tabular}{|l|c|c|c|c|}
\hline 
Type/Space  & dim  & Eigenvalues/spaces  & $\chi_{G/H}\left(t\right)$  & $\chi$\tabularnewline
\hline 
\begin{cellvarwidth}[t]
\medskip{}

$\mathrm{A}_{n}\mathrm{III}_{1}$: $\mathbb{CP}^{n}$ 
\end{cellvarwidth} & \begin{cellvarwidth}[t]
\medskip{}

$2n$ 
\end{cellvarwidth} & \begin{cellvarwidth}[t]
\medskip{}

$\lambda_{\mathfrak{u}\left(1\right)}=\frac{1}{2}$,

$\lambda_{\mathfrak{su}\left(n\right)}=\frac{1}{2}\frac{n}{n+1}$

\medskip{}
 
\end{cellvarwidth} & \begin{cellvarwidth}[t]
\medskip{}

$\left[n+1\right]_{t^{2}}$ 
\end{cellvarwidth} & \begin{cellvarwidth}[t]
\medskip{}

$n+1$
\end{cellvarwidth}\tabularnewline
\hline 
\begin{cellvarwidth}[t]
\medskip{}

$\mathrm{A}_{p+q-1}\mathrm{III}_{p}$: $\frac{\mathrm{U}\left(p+q\right)}{\mathrm{U}\left(p\right)\times\mathrm{U}\left(q\right)}$,

where $2\leq p\leq q$. 
\end{cellvarwidth} & \begin{cellvarwidth}[t]
\medskip{}

$2pq$ 
\end{cellvarwidth} & \begin{cellvarwidth}[t]
\medskip{}

$\lambda_{\mathfrak{u}\left(1\right)}=\frac{1}{2}$,

$\lambda_{\mathfrak{su}\left(p\right)}=\frac{1}{2}\frac{q}{p+q}$,

$\lambda_{\mathfrak{su}\left(q\right)}=\frac{1}{2}\frac{p}{p+q}$

\medskip{}
 
\end{cellvarwidth} & \begin{cellvarwidth}[t]
\medskip{}

$\binom{p+q}{p}_{t^{2}}$ 
\end{cellvarwidth} & \begin{cellvarwidth}[t]
\medskip{}

$\binom{p+q}{p}$
\end{cellvarwidth}\tabularnewline
\hline 
\begin{cellvarwidth}[t]
\medskip{}

$\mathrm{C}_{1+n}\mathrm{II}_{1}$: $\mathbb{HP}^{n}$ 
\end{cellvarwidth} & \begin{cellvarwidth}[t]
\medskip{}

$4n$ 
\end{cellvarwidth} & \begin{cellvarwidth}[t]
\medskip{}

$\lambda_{\mathfrak{sp}\left(1\right)}=\frac{1}{2}\frac{n}{n+2}$,

$\lambda_{\mathfrak{sp}\left(n\right)}=\frac{1}{2}\frac{1}{n+2}$

\medskip{}
 
\end{cellvarwidth} & \begin{cellvarwidth}[t]
\medskip{}

$\left[n+1\right]_{t^{4}}$ 
\end{cellvarwidth} & \begin{cellvarwidth}[t]
\medskip{}

$n+1$
\end{cellvarwidth}\tabularnewline
\hline 
\begin{cellvarwidth}[t]
\medskip{}

$\mathrm{C}_{p+q}\mathrm{II}_{p}$: $\frac{\mathrm{Sp}\left(p+q\right)}{\mathrm{Sp}\left(p\right)\times\mathrm{Sp}\left(q\right)}$,

where $2\leq p\leq q$. 
\end{cellvarwidth} & \begin{cellvarwidth}[t]
\medskip{}

$4pq$ 
\end{cellvarwidth} & \begin{cellvarwidth}[t]
\medskip{}

$\lambda_{\mathfrak{sp}\left(p\right)}=\frac{1}{2}\frac{q}{p+q+1}$,

$\lambda_{\mathfrak{sp}\left(q\right)}=\frac{1}{2}\frac{p}{p+q+1}$

\medskip{}
 
\end{cellvarwidth} & \begin{cellvarwidth}[t]
\medskip{}

$\binom{p+q}{p}_{t^{4}}$ 
\end{cellvarwidth} & \begin{cellvarwidth}[t]
\medskip{}

$\binom{p+q}{p}$
\end{cellvarwidth}\tabularnewline
\hline 
\begin{cellvarwidth}[t]
\medskip{}

$S^{n}=\frac{\mathrm{SO}\left(n+1\right)}{\mathrm{SO}\left(n\right)}$ 
\end{cellvarwidth} & \begin{cellvarwidth}[t]
\medskip{}

$n$ 
\end{cellvarwidth} & \begin{cellvarwidth}[t]
\medskip{}

$\lambda_{\mathfrak{so}\left(n\right)}=\frac{1}{2}\frac{1}{n-1}$ 
\end{cellvarwidth} & \begin{cellvarwidth}[t]
\medskip{}

$1+t^{n}$ 
\end{cellvarwidth} & \begin{cellvarwidth}[t]
\medskip{}

$1+\left(-1\right)^{n}$
\end{cellvarwidth}\tabularnewline
\hline 
\begin{cellvarwidth}[t]
\medskip{}

$\mathrm{BDI}$: $\frac{\mathrm{SO}\left(p+q\right)}{\mathrm{SO}\left(p\right)\times\mathrm{SO}\left(q\right)}$,

where $2\leq p\leq q$. 
\end{cellvarwidth} & \begin{cellvarwidth}[t]
\medskip{}

$pq$ 
\end{cellvarwidth} & \begin{cellvarwidth}[t]
\medskip{}

$\lambda_{\mathfrak{so}(p)}=\frac{1}{2}\frac{q}{p+q-2}$,

$\lambda_{\mathfrak{so}(q)}=\frac{1}{2}\frac{p}{p+q-2}$

\medskip{}
 
\end{cellvarwidth} &  & \tabularnewline
\hline 
\begin{cellvarwidth}[t]
\medskip{}

$\mathrm{B}_{k+l}\mathrm{I}_{2k}$, 
\end{cellvarwidth} & \begin{cellvarwidth}[t]
\medskip{}

$2k\left(2l+1\right)$ 
\end{cellvarwidth} & \begin{cellvarwidth}[t]
\medskip{}

$\lambda_{\mathfrak{so}(2k)}=\frac{1}{2}\frac{2l+1}{2k+2l-1}$,

$\lambda_{\mathfrak{so}(2l+1)}=\frac{1}{2}\frac{2k}{2k+2l-1}$

\medskip{}
 
\end{cellvarwidth} & \begin{cellvarwidth}[t]
\medskip{}

$\binom{k+l}{k}_{t^{4}}\left[2\right]_{t^{2k}}$ 
\end{cellvarwidth} & \begin{cellvarwidth}[t]
\medskip{}

$2\binom{k+l}{k}$
\end{cellvarwidth}\tabularnewline
\hline 
\begin{cellvarwidth}[t]
\medskip{}

$\mathrm{D}_{k+l+1}\mathrm{I}_{2k+1}$ 
\end{cellvarwidth} & \begin{cellvarwidth}[t]
\medskip{}

$\left(2k+1\right)\left(2l+1\right)$ 
\end{cellvarwidth} & \begin{cellvarwidth}[t]
\medskip{}

$\lambda_{\mathfrak{so}(2k+1)}=\frac{1}{2}\frac{2l+1}{2k+2l}$,

$\lambda_{\mathfrak{so}(2l+1)}=\frac{1}{2}\frac{2k+1}{2k+2l}$

\medskip{}
 
\end{cellvarwidth} & \begin{cellvarwidth}[t]
\medskip{}

$\left(1+t^{2k+2l+1}\right)$

$\times\binom{k+l}{k}_{t^{4}}$ 
\end{cellvarwidth} & \begin{cellvarwidth}[t]
\medskip{}

0
\end{cellvarwidth}\tabularnewline
\hline 
\begin{cellvarwidth}[t]
\medskip{}

$\mathrm{D}_{k+l}\mathrm{I}_{2k}$ 
\end{cellvarwidth} & \begin{cellvarwidth}[t]
\medskip{}

$4kl$ 
\end{cellvarwidth} & \begin{cellvarwidth}[t]
\medskip{}
$\lambda_{\mathfrak{so}(2k)}=\frac{1}{2}\frac{l}{k+l-1}$

$\lambda_{\mathfrak{so}(2l)}=\frac{1}{2}\frac{k}{k+l-1}$,

\medskip{}
 
\end{cellvarwidth} & \begin{cellvarwidth}[t]
\medskip{}

$\frac{\left[2\right]_{t^{2k}}\left[2\right]_{t^{2l}}}{\left[2\right]_{t^{2k+2l}}}$

$\times\binom{k+l}{k}_{t^{4}}$ 
\end{cellvarwidth} & \begin{cellvarwidth}[t]
\medskip{}

$2\binom{k+l}{k}$
\end{cellvarwidth}\tabularnewline
\hline 
\end{tabular}%
\end{minipage}\newpage{}

\subsection{The Symmetric Spaces with Simple Isotropy:}\label{subsec:Simple-isotropy}

Aside from Lie groups there are several symmetric spaces that have
simple isotropy group. In this case we quickly obtain a generalization
of the formula for Lie groups: 
\[
\lambda_{\mathfrak{h}}=\frac{1}{4}\frac{\dim M}{\dim\mathfrak{h}}.
\]
Unlike for Lie groups this eigenvalue is never $1$/4 as $\dim M\neq\dim\mathfrak{h}$
in all other cases. We include one exception, $\mathrm{SU}\left(4\right)/\mathrm{SO}\left(4\right)$, as it belong to a family where all other spaces have simple isotropy group. Here the holonomy is not simple but the formula still holds and can be recalculated using Killing forms.

This class also contains the most difficult to calculate Poincar\'{e}
polynomials as all the remaining cases with unequal rank appear here.
Borel calculated some of these cases and Takeuchi calculated the remaining
ones with his formula from above. One notable feature of the Betti
numbers is that $b_{2}=b_{3}=b_{4}=0$ in all cases. 

There are four situations where the ranks are not equal, with one needing a further subdivision. We will illustrate
two of these here including the case with a subdivision.

We start with: 
\[
\frac{\mathrm{SU}\left(2m+1\right)}{\mathrm{SO}\left(2m+1\right)},\,m\geq1,
\]
where 
\begin{eqnarray*}
D_{\mathrm{SU}\left(2m+1\right)} & = & \left\{ 2,3,...,2m+1\right\} ,\\
D_{\mathrm{SO}\left(2m+1\right)} & = & \left\{ 2,4,...,2m\right\} .
\end{eqnarray*}
Here all the polynomials with odd degrees in $D_{\mathrm{SU}\left(2m-1\right)}$
will vanish when restricted. At the same time all of the even degrees
match precisely with the even degrees in $D_{\mathrm{SO}\left(2m-1\right)}$
so we obtain 
\[
\chi\left(t\right)=\left(1+t^{5}\right)\left(1+t^{9}\right)\cdots\left(1+t^{4m+1}\right).
\]

It is worth noting that $\frac{\mathrm{SU}\left(1\right)}{\mathrm{SO}\left(1\right)}=S^{3}$
is not included, but is a noteworthy example, also $\frac{\mathrm{SU}\left(3\right)}{\mathrm{SO}\left(3\right)}$
is of interest as it is a rational 5-sphere. However, the long exact
homotopy sequence shows that $\pi_{2}\left(\frac{\mathrm{SU}\left(3\right)}{\mathrm{SO}\left(3\right)}\right)=\mathbb{Z}_{2}$
so it is not the 5-sphere. It is the only example of an irreducible
symmetric space that is a rational homology sphere without being a
sphere (see also \cite{Wolf1969} for the complete classification).

In the other situation: 
\[
\frac{\mathrm{SU}\left(2m\right)}{\mathrm{SO}\left(2m\right)},\,m\geq2,
\]
we have 
\begin{eqnarray*}
D_{\mathrm{SU}\left(2m\right)} & = & \left\{ 2,3,...,2m\right\} ,\\
D_{\mathrm{SO}\left(2m\right)} & = & \left\{ m,2,4,...,2\left(m-1\right)\right\} .
\end{eqnarray*}
Here again all of the polynomials with odd degrees in $D_{\mathrm{SU}\left(2m\right)}$
vanish when restricted. This again leaves all of the even degrees
and all but $2m$ will be matched in $D_{\mathrm{SO}\left(2m\right)}$.
Here the polynomial of degree $m$ in $D_{\mathrm{SU}\left(2m\right)}$
could potentially be matched with the degree $m$ from $D_{\mathrm{SO}\left(2m\right)}$,
but this doesn't happen as the degree $m$ polynomial for $\mathrm{SO}\left(2m\right)$
is of a different type related to the horizontal flip in the Dynkin
diagram. We obtain 
\begin{eqnarray*}
\chi\left(t\right) & = & \frac{\left[2m\right]_{t^{2}}}{\left[m\right]_{t^{2}}}\left(1+t^{5}\right)\left(1+t^{9}\right)\cdots\left(1+t^{4m-3}\right)\\
 & = & \left[2\right]_{t^{2m}}\left(1+t^{5}\right)\left(1+t^{9}\right)\cdots\left(1+t^{4m-3}\right).
\end{eqnarray*}

Finally, we consider 
\[
\frac{\mathrm{E}_{6}}{\mathrm{Sp}\left(4\right)/\left\{ \pm I\right\} }
\]
with 
\begin{eqnarray*}
D_{\mathrm{E}_{6}} & = & \left\{ 2,5,6,8,9,12\right\} ,\\
D_{\mathrm{Sp}\left(4\right)} & = & \left\{ 2,4,6,8\right\} .
\end{eqnarray*}
Here again the odd degrees for $\mathrm{E}_{6}$ correspond to polynomials
that vanish when restricted so we get: 
\begin{eqnarray*}
\chi\left(t\right) & = & \frac{\left[12\right]_{t^{2}}}{\left[4\right]_{t^{2}}}\left(1+t^{9}\right)\left(1+t^{17}\right)\\
 & = & \left[3\right]_{t^{8}}\left(1+t^{9}\right)\left(1+t^{17}\right).
\end{eqnarray*}

\newpage{}

\medskip{}

\noindent{}%
\noindent\begin{minipage}[t]{1\columnwidth}%
\begin{tabular}{|l|c|c|c|c|}
\hline 
Type/Space  & dim  & Eigenvalues/spaces  & $\chi_{G/H}\left(t\right)$  & $\chi$\tabularnewline
\hline 
\begin{cellvarwidth}[t]
\medskip{}

$\mathrm{A}_{n-1}\mathrm{I}$: $\frac{\mathrm{SU}\left(n\right)}{\mathrm{SO}\left(n\right)}$,
$n\geq3$.

\medskip{}
 
\end{cellvarwidth} & \begin{cellvarwidth}[t]
\centering
 \medskip{}

$\frac{\left(n-1\right)\left(n+2\right)}{2}$ 
\end{cellvarwidth} & \begin{cellvarwidth}[t]
\medskip{}

$\lambda_{\mathfrak{so}\left(n\right)}=\frac{1}{4}\frac{n+2}{n}$ 
\end{cellvarwidth} &  & \tabularnewline
\hline 
\begin{cellvarwidth}[t]
\medskip{}

$\mathrm{A}_{2m-1}\mathrm{I}$

\medskip{}
 
\end{cellvarwidth} & \begin{cellvarwidth}[t]
\centering
 \medskip{}

$\left(2m-1\right)\left(m+1\right)$ 
\end{cellvarwidth} & \begin{cellvarwidth}[t]
\medskip{}

$\lambda_{\mathfrak{so}\left(2m\right)}=\frac{1}{4}\frac{m+1}{m}$ 
\end{cellvarwidth} & \begin{cellvarwidth}[t]
\centering
 \medskip{}

$\left(1+t^{2m}\right)$

$\times\prod_{i=1}^{m-1}\left(1+t^{4i+1}\right)$

\medskip{}
 
\end{cellvarwidth} & \begin{cellvarwidth}[t]
\centering
 \medskip{}

0 
\end{cellvarwidth}\tabularnewline
\hline 
\begin{cellvarwidth}[t]
\medskip{}

$\mathrm{A}_{2m}\mathrm{I}$

\medskip{}
 
\end{cellvarwidth} & \begin{cellvarwidth}[t]
\centering
 \medskip{}

$\left(m-1\right)\left(2m+1\right)$ 
\end{cellvarwidth} & \begin{cellvarwidth}[t]
\medskip{}

$\lambda_{\mathfrak{so}\left(2m+1\right)}=\frac{1}{4}\frac{2m+3}{2m+1}$ %\frac{1}{4}\frac{2m+1}{2m-1} 
\end{cellvarwidth} & \begin{cellvarwidth}[t]
\centering
 \medskip{}

$\prod_{i=1}^{m-1}\left(1+t^{4i+1}\right)$ 
\end{cellvarwidth} & \begin{cellvarwidth}[t]
\centering
 \medskip{}

0 
\end{cellvarwidth}\tabularnewline
\hline 
\begin{cellvarwidth}[t]
\medskip{}

$\mathrm{A}_{2n-1}\mathrm{II}$: $\frac{\mathrm{SU}\left(2n\right)}{\mathrm{Sp}\left(n\right)}$,
$n\geq2$.

\medskip{}
 
\end{cellvarwidth} & \begin{cellvarwidth}[t]
\centering
 \medskip{}

$\left(n-1\right)\left(2n+1\right)$ 
\end{cellvarwidth} & \begin{cellvarwidth}[t]
\medskip{}

$\lambda_{\mathfrak{sp}\left(n\right)}=\frac{1}{2}\frac{n-1}{2n}$ 
\end{cellvarwidth} & \begin{cellvarwidth}[t]
\centering
 \medskip{}

$\prod_{i=1}^{n-1}\left(1+t^{4i+1}\right)$ 
\end{cellvarwidth} & \begin{cellvarwidth}[t]
\centering
 \medskip{}

0 
\end{cellvarwidth}\tabularnewline
\hline 
\begin{cellvarwidth}[t]
\medskip{}

$\mathrm{E}_{6}\mathrm{I}$: $\frac{\mathrm{E}_{6}}{\mathrm{Sp}\left(4\right)/\left\{ \pm I\right\} }$

\medskip{}
 
\end{cellvarwidth} & \begin{cellvarwidth}[t]
\centering
 \medskip{}

42 
\end{cellvarwidth} & \begin{cellvarwidth}[t]
\medskip{}

$\lambda_{\mathfrak{sp}\left(4\right)}=\frac{7}{24}$ 
\end{cellvarwidth} & \begin{cellvarwidth}[t]
\centering
 \medskip{}

$\left[3\right]_{t^{8}}\left(1+t^{9}\right)\left(1+t^{17}\right)$ 
\end{cellvarwidth} & \begin{cellvarwidth}[t]
\centering
 \medskip{}

0 
\end{cellvarwidth}\tabularnewline
\hline 
\begin{cellvarwidth}[t]
\medskip{}

$\mathrm{E}_{6}\mathrm{IV}$: $\frac{\mathrm{E}_{6}}{\mathrm{F}_{4}}$

\medskip{}
 
\end{cellvarwidth} & \begin{cellvarwidth}[t]
\centering
 \medskip{}

26 
\end{cellvarwidth} & \begin{cellvarwidth}[t]
\medskip{}

$\lambda_{\mathfrak{f}_{4}}=\frac{1}{8}$ 
\end{cellvarwidth} & \begin{cellvarwidth}[t]
\centering
 \medskip{}

$\left(1+t^{9}\right)\left(1+t^{17}\right)$ 
\end{cellvarwidth} & \begin{cellvarwidth}[t]
\centering
 \medskip{}

0 
\end{cellvarwidth}\tabularnewline
\hline 
\begin{cellvarwidth}[t]
\medskip{}

$\mathrm{E}_{7}\mathrm{V}$: $\frac{\mathrm{E}_{7}}{\mathrm{SU}\left(8\right)/\left\{ \pm I\right\} }$

\medskip{}
 
\end{cellvarwidth} & \begin{cellvarwidth}[t]
\centering
 \medskip{}

70 
\end{cellvarwidth} & \begin{cellvarwidth}[t]
\medskip{}

$\lambda_{\mathfrak{su}\left(8\right)}=\frac{5}{18}$ 
\end{cellvarwidth} & \begin{cellvarwidth}[t]
\centering
 \medskip{}

$\left[6\right]_{t^{6}}\left[3\right]_{t^{8}}\left[2\right]_{t^{10}}\left[2\right]_{t^{14}}$ 
\end{cellvarwidth} & \begin{cellvarwidth}[t]
\centering
 \medskip{}

72 
\end{cellvarwidth}\tabularnewline
\hline 
\begin{cellvarwidth}[t]
\medskip{}

$\mathrm{E}_{8}\mathrm{VIII}$: $\frac{\mathrm{E}_{8}}{\mathrm{Spin}\left(16\right)/\left\{ \pm\mathrm{vol}\right\} }$

\medskip{}
 
\end{cellvarwidth} & \begin{cellvarwidth}[t]
\centering
 \medskip{}

128 
\end{cellvarwidth} & \begin{cellvarwidth}[t]
\medskip{}

$\lambda_{\mathfrak{so}\left(16\right)}=\frac{4}{15}$ 
\end{cellvarwidth} & \begin{cellvarwidth}[t]
\centering
 \medskip{}

$\left[5\right]_{t^{8}}\left[3\right]_{t^{12}}\left[3\right]_{t^{16}}\left[3\right]_{t^{20}}$ 
\end{cellvarwidth} & \begin{cellvarwidth}[t]
\centering
 \medskip{}

135 
\end{cellvarwidth}\tabularnewline
\hline 
\begin{cellvarwidth}[t]
\medskip{}

$\mathrm{F}_{4}\mathrm{II}$: $\frac{\mathrm{F}_{4}}{\mathrm{Spin}\left(9\right)}=\mathbb{OP}^{2}$

\medskip{}
 
\end{cellvarwidth} & \begin{cellvarwidth}[t]
\centering
 \medskip{}

16 
\end{cellvarwidth} & \begin{cellvarwidth}[t]
\medskip{}

$\lambda_{\mathfrak{so}\left(9\right)}=\frac{1}{9}$ 
\end{cellvarwidth} & \begin{cellvarwidth}[t]
\centering
 \medskip{}

$\left[3\right]_{t^{8}}$ 
\end{cellvarwidth} & \begin{cellvarwidth}[t]
\centering
 \medskip{}

3 
\end{cellvarwidth}\tabularnewline
\hline 
\end{tabular}%
\end{minipage}\newpage{}

\subsection{The Hermitian Symmetric Spaces}\label{subsec:Hermitian-Symmetric}

These spaces all have a 1-dimensional center
for the holonomy which is responsible for the maximal possible eigenvalue
$\frac12$. Aside from the complex Grassmannians there is only one additional
factor. These Hermitian spaces have
\begin{eqnarray*}
\lambda_{\mathfrak{u}\left(1\right)}=\lambda _{\mathfrak{h}_0} & = & \frac{1}{2},\\
\lambda_{\mathfrak{h}_{1}} & = & \frac{1}{4}\cdot\frac{\dim_{\mathbb R}M-2}{\dim\mathfrak{h}_{1}}.
\end{eqnarray*}
% \[
% \mathcolor{red}{\frac 12 = \rho = \frac{\operatorname{scal}}{2n} = \frac{\operatorname{tr}(\mathcal R)}{n} = \frac 1n (\frac 12 + \lambda \dim \mathfrak h_2) \implies \lambda = \frac 12 \frac{n- 1}{\dim \mathfrak h_2}}
% \]
Some noteworthy observations are that these spaces all have the maximal
possible eigenvalue $\frac12$ and with multiplicity 1. This is mirrored
in the fact that $b_{2}=1$ for Hermitian symmetric spaces and $b_{2}=0$
for all other compact irreducible symmetric spaces.  
We note that four examples:
\[\mathbb{CP}^{2}, \frac{\mathrm{SO}\left(8\right)}{\mathrm{SO}\left(2\right)\times\mathrm{SO}\left(6\right)}, \frac{\mathrm{E}_{6}}{\left(\mathrm{SO}\left(2\right)\times\mathrm{Spin}\left(10\right)\right)/\mathbb{Z}_{4}},\frac{\mathrm{E}_{7}}{\left(\mathrm{SO}\left(2\right)\times\mathrm{E}_{6}\right)/\mathbb{Z}_{3}}\]
have the same pair of eigenvalues $\frac 1   2, \frac{1}{6}$ but with distinct multiplicities for $\frac{1}{6}$.

In all cases Borel's formula can be used to
calculate the Poincar\'{e} polynomials. We feature the complex quadrics
as they have some interesting cohomological properties.

First the odd dimensional complex quadric: 
\[
\frac{\mathrm{SO}\left(2+2m+1\right)}{\mathrm{SO}\left(2\right)\times\mathrm{SO}\left(2m+1\right)}
\]
with 
\begin{eqnarray*}
\chi\left(t\right) & = & \binom{1+m}{1}_{t^{4}}\left[2\right]_{t^{2}}=\left[m+1\right]_{t^{4}}\left[2\right]_{t^{2}}=\left[2m+2\right]_{t^{2}}.
\end{eqnarray*}
It is noteworthy that this shows that this space has the rational
cohomology of $\mathbb{C}\mathbb{P}^{2m+1}$.

Next the even dimensional complex quadric which has an extra cohomology
class in the middle dimension:

\[
\frac{\mathrm{SO}\left(2+2m\right)}{\mathrm{SO}\left(2\right)\times\mathrm{SO}\left(2m\right)}
\]
with 
\[
\chi\left(t\right)=\binom{1+m}{1}_{t^{4}}\frac{\left[2\right]_{t^{2}}\left[2\right]_{t^{2m}}}{\left[2\right]_{t^{2+2m}}}=\frac{\left[2m+2\right]_{t^{2}}}{\left[2\right]_{t^{2+2m}}}\left[2\right]_{t^{2m}}=\left[m+1\right]_{t^{2}}\left[2\right]_{t^{2m}}.
\]

\newpage{}

\medskip{}

\noindent{}%
\noindent\begin{minipage}[t]{1\columnwidth}%
\begin{tabular}{|l|c|c|c|c|}
\hline 
Type/Space  & dim  & Eigenvalues/spaces  & $\chi_{G/H}\left(t\right)$  & $\chi$\tabularnewline
\hline 
\begin{cellvarwidth}[t]
\medskip{}

$\mathrm{A}_{n}\mathrm{III}_{1}$: $\mathbb{CP}^{n}$ 
\end{cellvarwidth} & \begin{cellvarwidth}[t]
\centering
 \medskip{}

$2n$ 
\end{cellvarwidth} & \begin{cellvarwidth}[t]
\medskip{}

$\lambda_{\mathfrak{u}\left(1\right)}=\frac{1}{2}$,

$\lambda_{\mathfrak{su}\left(n\right)}=\frac{1}{2}\frac{1}{n+1}$

\medskip{}
 
\end{cellvarwidth} & \begin{cellvarwidth}[t]
\centering
 \medskip{}

$\left[n+1\right]_{t^{2}}$ 
\end{cellvarwidth} & \begin{cellvarwidth}[t]
\centering
 \medskip{}

$n+1$ 
\end{cellvarwidth}\tabularnewline
\hline 
\begin{cellvarwidth}[t]
\medskip{}

$\mathrm{A}_{p+q-1}\mathrm{III}_{p}$: $\frac{\mathrm{U}\left(p+q\right)}{\mathrm{U}\left(p\right)\times\mathrm{U}\left(q\right)}$

where $2\leq p\leq q.$ 
\end{cellvarwidth} & \begin{cellvarwidth}[t]
\centering
 \medskip{}

$2pq$ 
\end{cellvarwidth} & \begin{cellvarwidth}[t]
\medskip{}

$\lambda_{\mathfrak{u}\left(1\right)}=\frac{1}{2}$,

$\lambda_{\mathfrak{su}\left(p\right)}=\frac{1}{2}\frac{q}{p+q}$,

$\lambda_{\mathfrak{su}\left(q\right)}=\frac{1}{2}\frac{p}{p+q}$

\medskip{}
 
\end{cellvarwidth} & \begin{cellvarwidth}[t]
\centering
 \medskip{}

$\binom{p+q}{p}_{t^{2}}$ 
\end{cellvarwidth} & \begin{cellvarwidth}[t]
\centering
 \medskip{}

$\binom{p+q}{p}$ 
\end{cellvarwidth}\tabularnewline
\hline 
\begin{cellvarwidth}[t]
\medskip{}

$\mathrm{BDI}$: $\frac{\mathrm{SO}\left(2+n\right)}{\mathrm{SO}\left(2\right)\times\mathrm{SO}\left(n\right)}$,

where $n\geq2.$ 
\end{cellvarwidth} & \begin{cellvarwidth}[t]
\centering
 \medskip{}

$2n$ 
\end{cellvarwidth} & \begin{cellvarwidth}[t]
\medskip{}

$\lambda_{\mathfrak{so}(2)}=\frac{1}{2}$,

$\lambda_{\mathfrak{so}(n)}=\frac{1}{n}$

\medskip{}
 
\end{cellvarwidth} &  & \tabularnewline
\hline 
\begin{cellvarwidth}[t]
\medskip{}

$\mathrm{B}_{1+m}\mathrm{I}_{2}$ 
\end{cellvarwidth} & \begin{cellvarwidth}[t]
\centering
 \medskip{}

$2\left(2m+1\right)$ 
\end{cellvarwidth} & \begin{cellvarwidth}[t]
\medskip{}

$\lambda_{\mathfrak{so}(2)}=\frac{1}{2}$,

$\lambda_{\mathfrak{so}(2m+1)}=\frac{1}{2m+1}$

\medskip{}
 
\end{cellvarwidth} & \begin{cellvarwidth}[t]
\centering
 \medskip{}

$\left[2m+2\right]_{t^{2}}$ 
\end{cellvarwidth} & \begin{cellvarwidth}[t]
\centering
 \medskip{}

$2m+2$ 
\end{cellvarwidth}\tabularnewline
\hline 
\begin{cellvarwidth}[t]
\medskip{}

$\mathrm{D}_{1+m}\mathrm{I}_{2}$ 
\end{cellvarwidth} & \begin{cellvarwidth}[t]
\centering
 \medskip{}

$4m$ 
\end{cellvarwidth} & \begin{cellvarwidth}[t]
\medskip{}

$\lambda_{\mathfrak{so}(2)}=\frac{1}{2}$,

$\lambda_{\mathfrak{so}(2m)}=\frac{1}{2m}$

\medskip{}
 
\end{cellvarwidth} & \begin{cellvarwidth}[t]
\centering
 \medskip{}

$\left[2\right]_{t^{2m}}\left[m+1\right]_{t^{2}}$ 
\end{cellvarwidth} & \begin{cellvarwidth}[t]
\centering
 \medskip{}

$2m+2$ 
\end{cellvarwidth}\tabularnewline
\hline 
\begin{cellvarwidth}[t]
\medskip{}

$\mathrm{C}_{n}\mathrm{I}$: $\frac{\mathrm{Sp}\left(n\right)}{\mathrm{U}\left(n\right)}$,
$n\geq1$. 
\end{cellvarwidth} & \begin{cellvarwidth}[t]
\centering
 \medskip{}

$n\left(n+1\right)$ 
\end{cellvarwidth} & \begin{cellvarwidth}[t]
\medskip{}

$\lambda_{\mathfrak{u}\left(1\right)}=\frac{1}{2}$,

$\lambda_{\mathfrak{su}\left(n\right)}= \frac{1}{4}\frac{n+2}{n+1}$ %\frac{1}{4}\frac{n}{n-1}-\frac{1}{2}\frac{1}{n^{2}-1}

\medskip{}
 
\end{cellvarwidth} & \begin{cellvarwidth}[t]
\centering
 \medskip{}

$\prod_{i=1}^{n}\left[2\right]_{t^{2i}}$ 
\end{cellvarwidth} & \begin{cellvarwidth}[t]
\centering
 \medskip{}

$2^{n}$ 
\end{cellvarwidth}\tabularnewline
\hline 
\begin{cellvarwidth}[t]
\medskip{}

$\mathrm{D}_{n}\mathrm{III}$: $\frac{\mathrm{SO}\left(2n\right)}{\mathrm{U}\left(n\right)}$,
$n\geq2$. 
\end{cellvarwidth} & \begin{cellvarwidth}[t]
\centering
 \medskip{}

$n\left(n-1\right)$ 
\end{cellvarwidth} & \begin{cellvarwidth}[t]
\medskip{}

$\lambda_{\mathfrak{u}\left(1\right)}=\frac{1}{2}$,

$\lambda_{\mathfrak{su}\left(n\right)}= \frac 14 \frac{n-2}{n-1}$ %\frac{1}{4}\frac{n}{n+1}-\frac{1}{2}\frac{1}{n^{2}-1}

\medskip{}
 
\end{cellvarwidth} & \begin{cellvarwidth}[t]
\centering
 \medskip{}

$\prod_{i=1}^{n-1}\left[2\right]_{t^{2i}}$ 
\end{cellvarwidth} & \begin{cellvarwidth}[t]
\centering
 \medskip{}

$2^{n-1}$ 
\end{cellvarwidth}\tabularnewline
\hline 
\begin{cellvarwidth}[t]
\medskip{}

$\mathrm{E}_{6}\mathrm{III}$:

\medskip{}

$\frac{\mathrm{E}_{6}}{\left(\mathrm{SO}\left(2\right)\times\mathrm{Spin}\left(10\right)\right)/\mathbb{Z}_{4}}$

\medskip{}
 
\end{cellvarwidth} & \begin{cellvarwidth}[t]
\centering
 \medskip{}

$32$ 
\end{cellvarwidth} & \begin{cellvarwidth}[t]
\medskip{}

$\lambda_{\mathfrak{so}\left(2\right)}=\frac{1}{2}$,

$\lambda_{\mathfrak{so}\left(10\right)}=\frac{1}{6}$ 
\end{cellvarwidth} & \begin{cellvarwidth}[t]
\centering
 \medskip{}

$\left[9\right]_{t^{2}}\left[3\right]_{t^{8}}$ 
\end{cellvarwidth} & \begin{cellvarwidth}[t]
\centering
 \medskip{}

27 
\end{cellvarwidth}\tabularnewline
\hline 
\begin{cellvarwidth}[t]
\medskip{}

$\mathrm{E}_{7}\mathrm{VII}$: $\frac{\mathrm{E}_{7}}{\left(\mathrm{SO}\left(2\right)\times\mathrm{E}_{6}\right)/\mathbb{Z}_{3}}$

\medskip{}
 
\end{cellvarwidth} & \begin{cellvarwidth}[t]
\centering
 \medskip{}

54 
\end{cellvarwidth} & \begin{cellvarwidth}[t]
\medskip{}

$\lambda_{\mathfrak{so}\left(2\right)}=\frac{1}{2}$,

$\lambda_{\mathfrak{e}_6}=\frac{1}{6}$ 
\end{cellvarwidth} & \begin{cellvarwidth}[t]
\centering
 \medskip{}

$\left[14\right]_{t^{2}}\left[2\right]_{t^{10}}\left[2\right]_{t^{18}}$ 
\end{cellvarwidth} & \begin{cellvarwidth}[t]
\centering
 \medskip{}

56 
\end{cellvarwidth}\tabularnewline
\hline 
\end{tabular}%
\end{minipage}

\newpage{}

\subsection{The Wolf Spaces}\label{subsec:Wolf-spaces}

These spaces were first classified by Joe Wolf (see \cite{Wolf1965}). They all have a three dimensional simple factor in the decomposition
of the holonomy with corresponding eigenvalue that is dictated by
the quaternionic structure and can be calculated on the quaternionic
projective space of the same dimension. Aside from Grassmannians 
there will only be one extra factor. For those Wolf spaces of
dimension $4n$ we have 
\begin{eqnarray*}
\lambda_{\mathfrak{sp}\left(1\right)} & = & \frac{1}{2}\cdot\frac{n}{n+2},\\
\lambda_{\mathfrak{h}_{2}} & = & \frac{1}{2}\cdot\frac{n}{n+2}\cdot\frac{2n+1}{\dim\mathfrak{h}_{2}}.
\end{eqnarray*}
It is tempting to think that the maximal eigenvalue is always $\frac{1}{2}\cdot\frac{n}{n+2}$
with multiplicity 3. This, however, fails in several ways. First of
all, there are several Wolf spaces which are also Hermitian symmetric
spaces, specifically: 
\[
\frac{\mathrm{SO}\left(6\right)}{\mathrm{SO}\left(4\right)\times\mathrm{SO}\left(2\right)}\,\,\mathrm{and}\,\,\frac{\mathrm{U}\left(2+n\right)}{\mathrm{U}\left(2\right)\times\mathrm{U}\left(n\right)}.\]

Next there is 
\[
\frac{\mathrm{SO}\left(8\right)}{\mathrm{SO}\left(4\right)\times\mathrm{SO}\left(4\right)}
\]
which only has one eigenvalue. Finally 
\[
\frac{\mathrm{G}_{2}}{\mathrm{SO}\left(4\right)}
\]
has two eigenvalues: $1/4$ with multiplicity 3 as expected, but the
other eigenvalue is $5/12$ also with multiplicity 3. 

Wolf spaces
all have $b_{4}\neq0$ but this does not characterize them as most Grassmannians
also have $b_{4}\neq0$. Note that the 9-dimensional example
\[
\frac{\mathrm{SU}\left(4\right)}{\mathrm{SO}\left(4\right)}
\]with Poincar\'{e} polynomial
\[\chi(t)=(1+t^4)(1+t^5) \]
and eigenvalue
\[\lambda = \frac{6}{16}=\frac{1}{2}\frac{6}{6+2}\]
is not a Wolf space despite having holonomy and an eigenvalue indicating that it could have a quaternionic structure.

In all cases Borel's formula can be used to calculate the Poincar\'{e}
polynomials. We point out that there are special things that happen
with the real Grassmannians of four planes. We start with the ``odd'' dimensional case: 
\[
\frac{\mathrm{SO}\left(4+2l+1\right)}{\mathrm{SO}\left(4\right)\times\mathrm{SO}\left(2l+1\right)}
\]
where 
\[
\chi\left(t\right)=\binom{2+l}{2}_{t^{4}}\left[2\right]_{t^{4}}=\left[l+1\right]_{t^{4}}\left[l+2\right]_{t^{4}}.
\]

The other case is a little less straightforward and shows some of
the complications with the formulas for real Grassmannians:

\[
\frac{\mathrm{SO}\left(4+2l\right)}{\mathrm{SO}\left(4\right)\times\mathrm{SO}\left(2l\right)}
\]
where 
\begin{eqnarray*}
\chi\left(t\right) & = & \binom{2+l}{2}_{t^{4}}\frac{\left[2\right]_{t^{4}}\left[2\right]_{t^{2l}}}{\left[2\right]_{t^{4+2l}}}=\left[l+1\right]_{t^{4}}\left[l+2\right]_{t^{4}}\frac{\left[2\right]_{t^{2l}}}{\left[2\right]_{t^{4+2l}}}\\
 & = & \left[l+1\right]_{t^{4}}\left[l+2\right]_{t^{4}}\frac{\left[2\right]_{t^{2l}}}{\left[2\right]_{t^{4+2l}}}\\
 & = & \left[l+1\right]_{t^{4}}\left(\frac{\left[2\right]_{t^{2l}}}{1-t^{4}}\frac{1-t^{4l+8}}{1+t^{2l+4}}\right)\\
 & = & \left[l+1\right]_{t^{4}}\left(\frac{1+t^{2l}}{1-t^{4}}\left(1-t^{2l+4}\right)\right)\\
 & = & \left[l+1\right]_{t^{4}}\left(\frac{1+t^{2l}-t^{2l+4}-t^{4l+4}}{1-t^{4}}\right)\\
 & = & \left[l+1\right]_{t^{4}}\left(\left[l+1\right]_{t^{4}}+t^{2l}\right).
\end{eqnarray*}

\newpage{}

\medskip{}

\noindent{}%
\noindent\begin{minipage}[t]{1\columnwidth}%
\begin{tabular}{|l|c|c|c|c|}
\hline 
Type/Space  & dim  & Eigenvalues/spaces  & $\chi_{G/H}\left(t\right)$  & $\chi$\tabularnewline
\hline 
\begin{cellvarwidth}[t]
\medskip{}

$\mathrm{C}_{1+n}\mathrm{II}_{1}$: $\mathbb{HP}^{n}$ 
\end{cellvarwidth} & \begin{cellvarwidth}[t]
\centering
 \medskip{}

$4n$ 
\end{cellvarwidth} & \begin{cellvarwidth}[t]
\medskip{}
 $\lambda_{\mathfrak{sp}\left(1\right)}=\frac{1}{2}\frac{n}{n+2}$,

$\lambda_{\mathfrak{sp}\left(n\right)}=\frac{1}{2}\frac{1}{n+2}$

%\medskip{}
 
\end{cellvarwidth} & \begin{cellvarwidth}[t]
\centering
 \medskip{}

$\left[n+1\right]_{t^{4}}$ 
\end{cellvarwidth} & \begin{cellvarwidth}[t]
\centering
 \medskip{}

$n+1$ 
\end{cellvarwidth}\tabularnewline
\hline 
\begin{cellvarwidth}[t]
\medskip{}

$\mathrm{A}_{n+1}\mathrm{III}_{2}$: $\frac{\mathrm{U}\left(2+n\right)}{\mathrm{U}\left(2\right)\times\mathrm{U}\left(n\right)}$,

where $n\geq2$. 
\end{cellvarwidth} & \begin{cellvarwidth}[t]
\centering
 \medskip{}

$4n$ 
\end{cellvarwidth} & \begin{cellvarwidth}[t]
\medskip{}
 $\lambda_{\mathfrak{u}\left(1\right)}=\frac{1}{2}$,

$\lambda_{\mathfrak{su}\left(2\right)}=\frac{1}{2}\frac{n}{n+2}$,

$\lambda_{\mathfrak{su}\left(n\right)}=\frac{1}{2}\frac{2}{n+2}$

%\medskip{}
 
\end{cellvarwidth} & \begin{cellvarwidth}[t]
\centering
 \medskip{}

$\frac{\left[n+1\right]_{t^{2}}\left[n+2\right]_{t^{2}}}{\left[2\right]_{t^{2}}}$ 
\end{cellvarwidth} & \begin{cellvarwidth}[t]
\centering
 \medskip{}

$\frac{\left(n+1\right)\left(n+2\right)}{2}$ 
\end{cellvarwidth}\tabularnewline
\hline 
\begin{cellvarwidth}[t]
\medskip{}

$\mathrm{A}_{2m+1}\mathrm{III}_{2}$ 
\end{cellvarwidth} & \begin{cellvarwidth}[t]
\centering
 \medskip{}

$8m$ 
\end{cellvarwidth} & \begin{cellvarwidth}[t]
\medskip{}
% $\lambda_{\mathfrak{u}\left(1\right)}=\frac{1}{2}$,

%$\lambda_{\mathfrak{su}\left(2\right)}=\frac{1}{2}\frac{m}{m+1}$,

%$\lambda_{\mathfrak{su}\left(2m\right)}=\frac{1}{2}\frac{1}{m+1}$

%\medskip{}
 
\end{cellvarwidth} & \begin{cellvarwidth}[t]
\centering
 \medskip{}

$\left[2m+1\right]_{t^{2}}\left[m+1\right]_{t^{4}}$ 
\end{cellvarwidth} & \begin{cellvarwidth}[t]
\centering
 \medskip{}

$\left(2m+1\right)$

$\times\left(m+1\right)$ 
\end{cellvarwidth}\tabularnewline
\hline 
\begin{cellvarwidth}[t]
\medskip{}

$\mathrm{A}_{2m+2}\mathrm{III}_{2}$ 
\end{cellvarwidth} & \begin{cellvarwidth}[t]
\centering
 \medskip{}

$4\left(2m+1\right)$ 
\end{cellvarwidth} & \begin{cellvarwidth}[t]
\medskip{}

\end{cellvarwidth} & \begin{cellvarwidth}[t]
\centering
 \medskip{}

$\left[m+1\right]_{t^{4}}\left[2m+3\right]_{t^{2}}$ 
\end{cellvarwidth} & \begin{cellvarwidth}[t]
\centering
 \medskip{}

$\left(m+2\right)$

$\times\left(2m+3\right)$ 
\end{cellvarwidth}\tabularnewline
\hline 
\begin{cellvarwidth}[t]
\medskip{}

$\mathrm{BDI}$: $\frac{\mathrm{SO}\left(4+n\right)}{\mathrm{SO}\left(4\right)\times\mathrm{SO}\left(n\right)}$,

where $n\geq2$. 
\end{cellvarwidth} & \begin{cellvarwidth}[t]
\centering
 \medskip{}

$4n$ 
\end{cellvarwidth} & \begin{cellvarwidth}[t]
\medskip{}
 $\lambda_{\mathfrak{so}(4)}=\frac{1}{2}\frac{n}{n+2}$,

$\lambda_{\mathfrak{so}(n)}=\frac{2}{n+2}$

%\medskip{}
 
\end{cellvarwidth} &  & \tabularnewline
\hline 

\medskip{}

 $\mathrm{B}_{2+l}\mathrm{I}_{4}$  & \begin{cellvarwidth}[t]
\centering
 \medskip{}

$4\left(2l+1\right)$ 
\end{cellvarwidth} & \begin{cellvarwidth}[t]
\medskip{}
 $\lambda_{\mathfrak{so}(4)}=\frac{1}{2}\frac{2l+1}{2l+3}$,

$\lambda_{\mathfrak{so}(2l+1)}=\frac{2}{2l+3}$

%\medskip{}
 
\end{cellvarwidth} & \begin{cellvarwidth}[t]
\centering
 \medskip{}

$\left[l+1\right]_{t^{4}}\left[l+2\right]_{t^{4}}$ 
\end{cellvarwidth} & \begin{cellvarwidth}[t]
\centering
 \medskip{}

$\left(l+1\right)$

$\times\left(l+2\right)$ 
\end{cellvarwidth}\tabularnewline
\hline 

\medskip{}

 $\mathrm{D}_{2+l}\mathrm{I}_{4}$  & \begin{cellvarwidth}[t]
\centering
 \medskip{}

$8l$ 
\end{cellvarwidth} & \begin{cellvarwidth}[t]
\medskip{}
 $\lambda_{\mathfrak{so}(4)}=\frac{1}{2}\frac{l}{l+1}$,

$\lambda_{\mathfrak{so}(2l)}=\frac{1}{l+1}$

%\medskip{}
 
\end{cellvarwidth} & \begin{cellvarwidth}[t]
\centering
 \medskip{}

$\left[l+1\right]_{t^{4}}$

$\times\left(t^{2l}+\left[l+1\right]_{t^{4}}\right)$ 
\end{cellvarwidth} & \begin{cellvarwidth}[t]
\centering
 \medskip{}

$\left(l+1\right)$

$\times\left(l+2\right)$ 
\end{cellvarwidth}\tabularnewline
\hline 
\begin{cellvarwidth}[t]
\medskip{}

$\mathrm{E}_{6}\mathrm{II}$: $\frac{\mathrm{E}_{6}}{\left(\mathrm{SU}\left(2\right)\times\mathrm{SU}\left(6\right)\right)/\mathbb{Z}_{2}}$ 
\end{cellvarwidth} & \begin{cellvarwidth}[t]
\centering
 \medskip{}

$40$ 
\end{cellvarwidth} & \begin{cellvarwidth}[t]
\medskip{}
 $\lambda_{\mathfrak{su}\left(2\right)}=\frac{10}{24}$,

$\lambda_{\mathfrak{su}\left(6\right)}=\frac{1}{4}$

%\medskip{}
 
\end{cellvarwidth} & \begin{cellvarwidth}[t]
\centering
 \medskip{}

$\left[6\right]_{t^{4}}\left[3\right]_{t^{6}}\left[2\right]_{t^{8}}$ 
\end{cellvarwidth} & \begin{cellvarwidth}[t]
\centering
 \medskip{}

$36$ 
\end{cellvarwidth}\tabularnewline
\hline 
\begin{cellvarwidth}[t]
\medskip{}

$\mathrm{E}_{7}\mathrm{VI}$: $\frac{\mathrm{E}_{7}}{\left(\mathrm{
SU}\left(2\right)\times\mathrm{Spin}\left(12\right)\right)/\mathbb{Z}_{2}}$ 
\end{cellvarwidth} & \begin{cellvarwidth}[t]
\centering
 \medskip{}

$64$ 
\end{cellvarwidth} & \begin{cellvarwidth}[t]
\medskip{}

$\lambda_{\mathfrak{su}\left(2\right)}=\frac{16}{36}$,

$\lambda_{\mathfrak{so}\left(12\right)}=\frac{2}{9}$

%\medskip{}
 
\end{cellvarwidth} & \begin{cellvarwidth}[t]
\centering
 \medskip{}

$\left[7\right]_{t^{4}}\left[3\right]_{t^{8}}\left[3\right]_{t^{12}}$ 
\end{cellvarwidth} & \begin{cellvarwidth}[t]
\centering
 \medskip{}

$63$ 
\end{cellvarwidth}\tabularnewline
\hline 
\begin{cellvarwidth}[t]
\medskip{}

$\mathrm{E}_{8}\mathrm{IX}$: $\frac{\mathrm{E}_{8}}{\left(\mathrm{SU}\left(2\right)\times\mathrm{E}_{7}\right)/\mathbb{Z}_{2}}$ 
\end{cellvarwidth} & \begin{cellvarwidth}[t]
\centering
 \medskip{}

$112$ 
\end{cellvarwidth} & \begin{cellvarwidth}[t]
\medskip{}

$\lambda_{\mathfrak{su}\left(2\right)}=\frac{28}{60}$,

$\lambda_{\mathfrak{e}_{7}}=\frac{1}{5}$

%\medskip{}
 
\end{cellvarwidth} & \begin{cellvarwidth}[t]
\centering
 \medskip{}

$\left[15\right]_{t^{4}}\left[4\right]_{t^{12}}\left[2\right]_{t^{20}}$ 
\end{cellvarwidth} & \begin{cellvarwidth}[t]
\centering
 \medskip{}

$120$ 
\end{cellvarwidth}\tabularnewline
\hline 
\begin{cellvarwidth}[t]
\medskip{}

$\mathrm{F}_{4}\mathrm{I}$: $\frac{\mathrm{F}_{4}}{\left(\mathrm{SU}\left(2\right)\times\mathrm{Sp}\left(3\right)\right)/\mathbb{Z}_{2}}$ 
\end{cellvarwidth} & \begin{cellvarwidth}[t]
\centering
 \medskip{}

$28$ 
\end{cellvarwidth} & \begin{cellvarwidth}[t]
\medskip{}

$\lambda_{\mathfrak{su}\left(2\right)}=\frac{7}{18}$,

$\lambda_{\mathfrak{sp}\left(3\right)}=\frac{5}{18}$

%\medskip{}
 
\end{cellvarwidth} & \begin{cellvarwidth}[t]
\centering
 \medskip{}

$\left[6\right]_{t^{4}}\left[2\right]_{t^{8}}$ 
\end{cellvarwidth} & \begin{cellvarwidth}[t]
\centering
 \medskip{}

$12$ 
\end{cellvarwidth}\tabularnewline
\hline 
\begin{cellvarwidth}[t]
\medskip{}

$\mathrm{G}_{2}\mathrm{I}$: $\frac{\mathrm{G}_{2}}{\mathrm{SO}\left(4\right)}$ 
\end{cellvarwidth} & \begin{cellvarwidth}[t]
\centering
 \medskip{}

$8$ 
\end{cellvarwidth} & \begin{cellvarwidth}[t]
\medskip{}

$\lambda_{\mathfrak{su}\left(2\right)}=\frac{1}{4}$,

$\lambda_{\mathfrak{su}\left(2\right)}=\frac{5}{12}$

%\medskip{}
 
\end{cellvarwidth} & \begin{cellvarwidth}[t]
\centering
 \medskip{}

$\left[3\right]_{t^{4}}$ 
\end{cellvarwidth} & \begin{cellvarwidth}[t]
\centering
 \medskip{}

$3$ 
\end{cellvarwidth}\tabularnewline
\hline 
\end{tabular}%
\end{minipage}

\bibliographystyle{abbrv}

\begin{thebibliography}{10}

\bibitem{Borel1953}
A.~Borel.
\newblock Sur la cohomologie des espaces fibr\'{e}s principaux et des espaces homog\`{e}nes de groupes de lie compacts.
\newblock {\em Annals of Mathematics}, 57(1):115--207, 1953.

\bibitem{Borel1960}
A.~Borel.
\newblock On the curvature tensor of the hermitian symmetric manifolds.
\newblock {\em Annals of Mathematics}, 71(3):508--521, 1960.

\bibitem{BorelChevalley1955}
A.~Borel and C.~Chevalley.
\newblock The betti numbers of the exceptional groups.
\newblock {\em Memoirs of the American Mathematical Society}, 14:1--9, 1955.

\bibitem{CV}
E.~Calabi and E.~Vesentini.
\newblock On compact, locally symmetric {K}\"ahler manifolds.
\newblock {\em Annals of Mathematics}, 71(2):472--507, 1960.

\bibitem{Coleman1958}
A.~J. Coleman.
\newblock The betti numbers of the simple lie groups.
\newblock {\em Canadian Journal of Mathematics}, 10:349--356, 1958.

\bibitem{GHV}
W.~Greub, S.~Halperin, and R.~Vanstone.
\newblock {\em Connections, Curvature and cohomology I,II,III}.
\newblock Academic Press, London (1972-73-76).

\bibitem{Iwamoto1949}
H.~Iwamoto.
\newblock On integral invariants and {B}etti numbers of symmetric {R}ienmannian manifolds, {I}.
\newblock {\em Journal of the Mathematical Society of Japan}, 1(2):91--110, 1949.

\bibitem{Koiso1978}
N.~Koiso.
\newblock Rigidity and stability of {E}instein metrcs--the case of compact symmetric spaces.
\newblock {\em Osaka Journal of Mathematics}, 17:51--73, 1980.

\bibitem{Kuin1965}
D.~Kuin'.
\newblock Poincar\'{e} polynomials of certain compact homogeneous spaces.
\newblock {\em Doklady Akademii Nauk SSSR}, 160(3):515--518, 1965.
\newblock in Russian.

\bibitem{Leung2016}
W.~H. Leung.
\newblock Cohomology of symmetric space.
\newblock M.phil. thesis, Hong Kong University of Science and Technology, 2016.

\bibitem{MinuraToda1991}
M.~Mimura and H.~Toda.
\newblock {\em Topology of Lie Groups, I and II.}
\newblock Amer. Math. Soc., 1991.

\bibitem{Ozawa2022}
T.~Ozawa.
\newblock The $\mathbb{Z}_2$-{B}etti numbers of oriented {G}rassmannians.
\newblock {\em Osaka Journal of Mathematics}, 59:843--851, 2022.

\bibitem{Piccinni2017}
P.~Piccinni.
\newblock On the cohomology of some exceptional symmetric spaces.
\newblock In {\em Special metrics and group actions in geometry}, pages 291--305. Springer, 2017.

\bibitem{Ramanujam1969}
S.~Ramanujam.
\newblock Morse theory of certain symmetric spaces.
\newblock {\em Journal of Differential Geometry}, 3(1-2):213--229, 1969.

\bibitem{Takeuchi1962}
M.~Takeuchi.
\newblock On the {P}ontrjagen classes of compact symmetric spaces.
\newblock {\em Journal of the Faculty of Science, University of Tokyo, Section I}, 9:313--328, 1962.

\bibitem{Terzic2001FPM}
S.~Terzi\v{c}.
\newblock Cohomology with real coefficients of generalized symmetric spaces.
\newblock {\em Fundamentalnaya i Prikladnaya Matematika}, 7(1):131--157, 2001.
\newblock in Russian.

\bibitem{Terzic2001}
S.~Terzi\v{c}.
\newblock {P}ontryagin classes of generalized symmetric spaces.
\newblock {\em Mathematical Notes}, 69(4):559--566, 2001.

\bibitem{Wang1949}
H.-C. Wang.
\newblock Homogeneous spaces with non-vanishing euler characteristics.
\newblock {\em Annals of Mathematics}, 50 no. 4:925--953, 1949.

\bibitem{Wolf1965}
J.~Wolf.
\newblock Homogeneous contact manifolds and quaternionic symmetric spaces.
\newblock {\em Journal of Mathematics and Mechanics}, 14 no. 6:1033--1047, 1965.

\bibitem{Wolf1969}
J.~Wolf.
\newblock Symmetric spaces which are real cohomology spheres.
\newblock {\em Journal of Differential Geometry}, 3:59--68, 1969.

\bibitem{Yang94}
K.~Yang.
\newblock Invariant k\"{a}hler metrics and projective embeddings of the flag manifold.
\newblock {\em Bulletin of the Australian Mathematical Society}, 49(2):239--247, 1994.

\bibitem{Yokota2025}
I.~Yokota.
\newblock {\em Exceptional Lie Groups}, volume 2369 of {\em Lecture Notes in Mathematics}.
\newblock Springer, 2025.

\end{thebibliography}

\end{document}